\documentclass{article}
\usepackage{amsmath,amssymb,amsthm,bm,enumerate}
\usepackage[dvipdfmx]{graphicx}
\usepackage[titletoc,title]{appendix}
\usepackage{epstopdf}
\usepackage{pdfpages}

\newcommand{\keywords}[1]{\textbf{Key words:} #1}
\newcommand{\msc}[1]{\textbf{MSC2010:} #1}

\DeclareMathOperator{\Supp}{Supp}

\def\rnum#1{\expandafter{\romannumeral #1}} 
\def\Rnum#1{\uppercase\expandafter{\romannumeral #1}}

\title{ Lamplighter random walks on fractals}
\date{October 6th, 2016}
\author{%
 Takashi Kumagai\thanks{Research Institute for Mathematical Sciences,
Kyoto University, Kyoto 606-8502, Japan. \newline ~~~~~E-mail: {\tt kumagai@kurims.kyoto-u.ac.jp}}
 \and 
 Chikara Nakamura\thanks{Research Institute for Mathematical Sciences,
Kyoto University, Kyoto 606-8502, Japan. \newline ~~~~~E-mail: {\tt chikaran@kurims.kyoto-u.ac.jp}}
}

\begin{document}

\newtheorem{definition}{Definition}[section]
\newtheorem{proposition}[definition]{Proposition}
\newtheorem{theorem}[definition]{Theorem}
\newtheorem{Assumption}[definition]{Assumption}
\newtheorem{lemma}[definition]{Lemma}
\newtheorem{remark}[definition]{Remark}
\newtheorem{Example}[definition]{Example}
\newtheorem{Corollary}[definition]{Corollary}
\newtheorem*{notation}{Notation}

\makeatletter
 \renewcommand{\theequation}{%
   \thesection.\arabic{equation}}
  \@addtoreset{equation}{chapter}
\makeatother

\maketitle


\begin{abstract}
We consider  on-diagonal heat kernel estimates and the laws of the iterated logarithm  
 for a switch-walk-switch random walk on a lamplighter graph 
 under the condition that 
 the random walk on the underlying graph enjoys sub-Gaussian heat kernel estimates. 

\end{abstract}

\begin{flushleft}
 \\
\msc{60J10; 60J35; 60J55}
\end{flushleft}

\section{Introduction} \label{Intro}
Let $G$ be a connected infinite graph and consider the situation that on each vertex of $G$ there is a lamp. 
Consider a lamplighter on the graph that makes the following random movements;
first, the lamplighter turns on or off the lamp on the site with equal probability, then he/she moves to the nearest neighbor of $G$ with equal probability, and  turns on or off the lamp on the new site with equal probability. The lamplighter repeats this random movement. Such a movement can be considered as a random 
walk on the wreath product of graphs $\mathbb{Z}_2 \wr G$ which is roughly a graph putting $\mathbb{Z}_2=\{0,1\}$ on each vertex of $G$ (see Definition \ref{Thm:wreath-P} for precise definition), and it is called a ``switch-walk-switch walk'' or ``lamplighter walk'' on $\mathbb{Z}_2 \wr G$. We are interested in the long time behavior of the walk. 
Some results are known when $G$ is a specific graph.   
Pittet and Saloff-Coste \cite{PS2} established on-diagonal heat kernel asymptotics  
of the random walk on $\mathbb{Z}_2 \wr \mathbb{Z}^d$.
More precisely, they obtained 
the following estimates; there exist positive constants $c_1 , c_2, c_3 , c_4 >0$ such that 
        \begin{gather}\label{eq:vares}
                 c_1  \exp \left[ - c_2 n^{ \frac{d}{d+2} } \right]   \le    h_{2n} (g,g)  \le c_3 \exp \left[ - c_4n^{ \frac{d}{d+2} } \right]  
        \end{gather} 
for all $g \in \mathbb{Z}_2 \wr \mathbb{Z}^d$, where $h_n(\cdot,\cdot)$ is the heat kernel
 (see  \cite{Varo1} \cite{Varo2} for earlier results  
 for the case of $G= \mathbb{Z}$ and \cite{CGP} for the case that $G$ is a finitely generated 
 group with polynomial volume growth). 
 Revelle \cite{Rev1} considered the lamplighter walk on the  wreath product $H \wr \mathbb{Z}$ 
 when $H$ is either a finite set or in a class of groups, and obtained some relations between the rate of escape of random walks on $H$ 
 and the law of the iterated logarithm (LIL in short) 
 on $H \wr \mathbb{Z}$. In particular, when $H=\mathbb{Z}_2$ he proved that  
there exist (non-random) constants $c_1 , c_2, c_3 , c_4 >0$ such that 
 the following hold for all $\bm{x} \in \mathbb{Z}_2 \wr \mathbb{Z}$:
       \begin{align}\label{eq:reves}
                c_1 \le  \limsup_{n \to \infty} \frac{ d(Y_0, Y_n) }{ n^{1/2}  (\log \log n)^{1/2} } 
                            \le c_2, ~~c_3 \le  \liminf_{n \to \infty}  \frac{  d (Y_0, Y_n)  }{ n^{1/2} (\log \log n)^{-1/2} } 
                            \le c_4, \quad 
                                 \text{ $P_{ \bm{x} }$-a.s.},   
       \end{align}  
 where $\{Y_n\}$ is the lamplighter random walk and $d(\cdot,\cdot)$ is   
 the graph distance on $\mathbb{Z}_2 \wr \mathbb{Z}$.\\

We are interested in the following question:

\medskip 
{\bf (Question)} How do the exponents in \eqref{eq:vares}, \eqref{eq:reves} change when 
the graph $G$ is more general? 

\begin{figure}[htb]
      \begin{minipage}{0.5\hsize}
                 \begin{center}
                            \includegraphics[width=50mm]{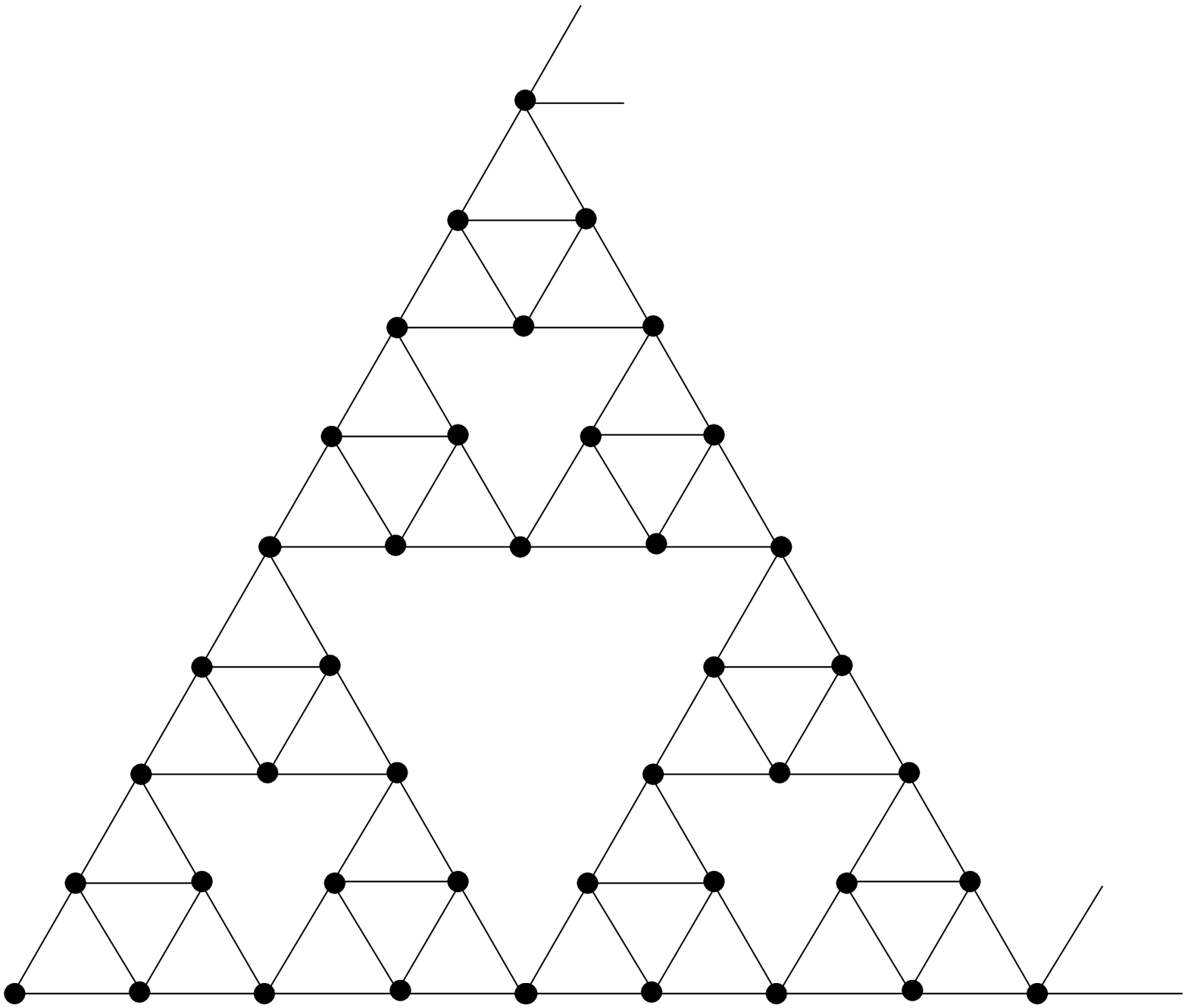}
                 \end{center}
                 \caption{The Sierpinski gasket graph. }
      \label{Fig:gasket}
     \end{minipage} 
     \begin{minipage}{0.5\hsize}
                \begin{center}
                            \includegraphics[width=50mm]{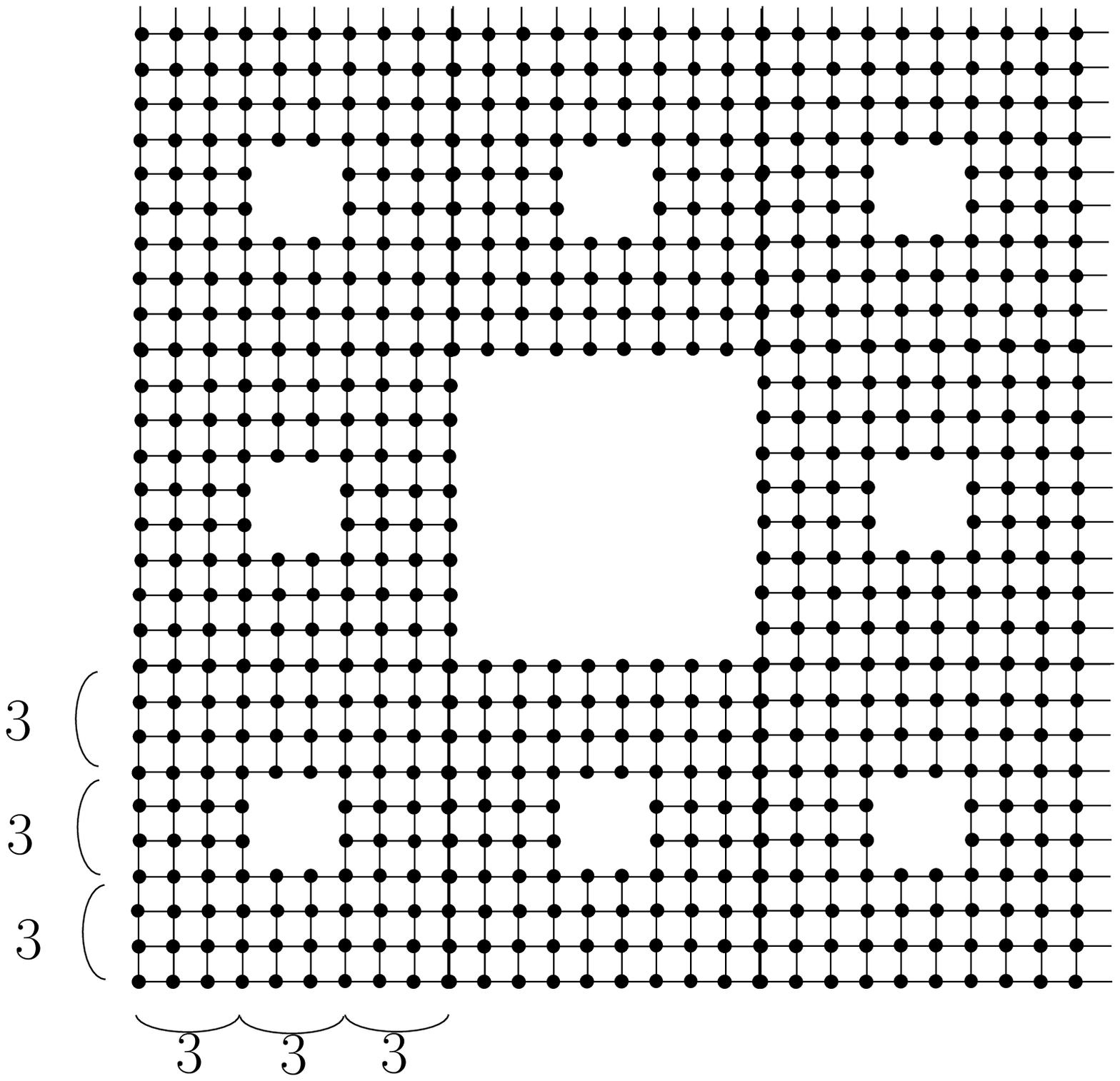}
                \end{center}
                \vspace{-5mm}
                \caption{The Sierpinski carpet graph. }
               \label{Fig:carpet}
     \end{minipage}
\end{figure}
 
\noindent 
In this paper, we will consider the question when $G$ is typically a fractal graph. Figure \ref{Fig:gasket} and Figure \ref{Fig:carpet}
illustrate concrete examples of fractal graphs. It is known that the random walk on such a fractal graph 
behaves anomalously in that it diffuses slower than 
a simple random walk on $\mathbb{Z}^d$.  It has been proved that the heat kernel $h_n (x,y)$ of the random walk $\{X_n \}_{n\ge 0}$ enjoys the following sub-Gaussian estimates; 
there exist positive constants $c_1 , c_2 , c_3 , c_4 >0$ such that
     \begin{align}   \label{eq:HKBM}
              & \frac{c_1}{n^{d_f/d_w}}   \exp \left( -c_2 \left( \frac{d(x,y)^{d_w} }{n} \right)^{1/(d_w-1)} \right)      \le h_{2n} (x,y)    \le   \frac{c_3}{n^{d_f/d_w}}   \exp \left( -c_4 \left( \frac{d(x,y)^{d_w} }{n} \right)^{1/(d_w-1)} \right) 
     \end{align} 
 for all $d(x,y)\le 2n$ (note that $h_{2n}(x,y)=0$ when $d(x,y)>2n$), where 
 $d(\cdot,\cdot)$ is the graph distance, $d_f$ is the volume growth exponent of the fractal graph and $d_w$ is called the walk dimension, 
 which expresses how fast the random walk on the fractal spreads out.  
 Indeed, by integrating  $(\ref{eq:HKBM})$, one can obtain the following estimates;  
 there exist positive constants $c_1 , c_2 >0$ such that 
     \begin{gather*}
                c_1 n^{1/d_w}  \le  E d(X_{2n}, X_0) \le c_2 n^{1/d_w}  
     \end{gather*} 
 for all $n >0$. 
 For more details on diffusions on fractals and random walks on fractal graphs, see \cite{Barlow1}, \cite{Kigami} and \cite{Kumagai2}. 
 As we see, properties of random walks on graphs are related to the geometric properties of the graphs.
 The volume growth is one of such properties. 
 For the graphs with  polynomial volume growth, there are well-established general methods to analyze the properties of random walks on them.
 But for the graphs with exponential volume growth, these methods are not applicable.   
 In this sense, the graphs with exponential volume growth  give us interesting research subject. 
 The wreath product $\mathbb{Z}_2 \wr G$ 
 belongs to this category,
  and this is another reason why we are interested in the lamplighter random walks on fractal graphs.\\

We consider the random walk on $\mathbb{Z}_2 \wr G$, where the 
 random walk on $G$ enjoys the sub-Gaussian heat kernel estimates $( \ref{eq:HKBM})$. 
   The main results of this paper are the following; 
                \begin{itemize}
                       \item[$(1)$] Sharp on-diagonal heat kernel estimates 
                                        for the random walk on $\mathbb{Z}_2 \wr G$ (Theorem \ref{Thm:HK}),
         
                       \item[$(2)$] LILs for the random walk on  $\mathbb{Z}_2 \wr G$  
                                       (Theorem \ref{Thm:LILrectra}).
                \end{itemize}

The on-diagonal heat kernel estimates are heavily related to the asymptotic properties of the spectrum of the corresponding discrete operator.   
We can obtain the exponent $d_f /(d_f + d_w)$ in our framework as the generalization of $d/(d+2)$.

For the LILs, we establish the LIL for $d_{\mathbb{Z}_2 \wr G} (Y_0, Y_n)$, 
 where $\{ Y_n \}_{n \ge 0}$ is the random walk on $\mathbb{Z}_2 \wr G$, 
 and the so-called another law of the iterated logarithm that gives the almost sure asymptotic behavior of the liminf of 
 $d_{\mathbb{Z}_2 \wr G} (Y_0, Y_n)$. 
 Note that in \eqref{eq:reves}, various properties that are specific to 
 $\mathbb{Z}$ were used, so the generalization to other graphs $G$  
 is highly non-trivial.  
We  have overcome the difficulty by finding some relationship between 
the range of the random walk on $G$ and $d_{\mathbb{Z}_2 \wr G} (Y_0, Y_n)$. 
To our knowledge, these are the first results on the LILs for the  wreath product beyond the case of $G=\mathbb{Z}$. \\

The outline of this paper is as follows. 
In section \ref{Sec:Notations}, we explain the framework and the main results of this paper.
In section \ref{Sec:ConseqHK}, we give some consequences of sub-Gaussian heat kernel estimates.
 These are preliminary results for section \ref{Sec:HK} and section \ref{Sec:LIL},
where we mainly treat the lamplighter random walks on fractal graphs.
In section \ref{Sec:HK}, we prove the on-diagonal heat kernel estimates. 
Section \ref{Sec:LIL} has three subsections. 
In subsection \ref{subsec:LIL41}, we give a relation between the range of random walk on $G$
  and $d_{\mathbb{Z}_2 \wr G} (Y_0, Y_n)$. 
Here, one of the keys is to prove the existence of a path that covers a subgraph of $G$ with the 
length of the path being (uniformly) comparable to the volume of the subgraph 
(Lemma \ref{Lem:LIL70}). 
In subsection \ref{subsec:LILrec},  we deduce the LILs for the random walk on $\mathbb{Z}_2 \wr G$ 
from the LILs for the range of the random walk on $G$ (Theorem \ref{Thm:LILdisc}) when $G$ is a strongly recurrent graph. 
In subsection \ref{subsec:LIL44}, we prove the LILs for the random walk on $\mathbb{Z}_2 \wr G$ when 
$G$ is a transient graph. 
In the Appendix \ref{sec:Appendix}, we give an outline of the proof of Theorem \ref{Thm:LILdisc}, on which the proof in subsection \ref{subsec:LILrec} is based. \\

Throughout this paper,  we use the following notation.
\begin{notation}

\begin{enumerate}
    \item[$(1)$]
          For two non-negative sequences $\{ a(n) \}_{n \ge 0 }$ and $\{ b(n) \}_{n \ge 0 }$,
         we write
             \begin{itemize}
                   \item    $a(n) \asymp  b(n)$ if there exist positive constants $c_1 , c_2>0$ such that 
                              $c_1 a(n) \le b (n)  \le c_2 a(n) $ for all $n$.  
                              
                    \item    $a(n) \approx b(n)$ if there exist positive constants $c_1 , c_2 , c_3 , c_4>0$ such that 
                               $c_1 a(c_2 n) \le b(n)  \le c_3 a(c_4 n)$ for all $n$.               \end{itemize}

    \item[$(2)$]
          We use $c, C, c_1 , c_2 , \cdots$ to denote deterministic positive finite constants whose values are insignificant.  
         These constants do not depend on time parameters $n,k, \cdots$,
          distance parameters $r,\cdots$, or 
         vertices of graphs. 
\end{enumerate}
 \end{notation}

\section{Framework and main results}   \label{Sec:Notations}
In this section, we introduce the framework and the main  results of this paper.

\subsection{Framework}     \label{subsec:Framework}  
Let $G=(V(G) , E(G) ) $ be an infinite, locally finite, connected graph.
 We assume $V(G)$ is a countable set.
 We say that $G$ is a graph of bounded degree if 
     \begin{gather}
                    M =   \sup_{ v \in V(G) } \deg v   < \infty .    \label{eq:bounded} 
     \end{gather} 
 We denote $d(x,y)$ the graph distance of $x,y$ in $G$, i.e. 
 the shortest length of paths between $x$ and $y$. 
 When we want to emphasize the graph $G$, we write $d_G (x,y)$ instead of $d(x,y)$.

Next, we introduce the wreath product of graphs. 

\begin{definition}[Wreath product] \label{Thm:wreath-P}
Let $G = (V(G),E(G))$ and $H=(V(H), E(H))$ be graphs.
We define the {\it wreath product of $G$ and $H$ (denoted by  $H \wr G$)} in the following way.
 We define the vertex set of the wreath product as 
    \begin{gather*}
              V ( H \wr G ) = \left\{ (f , v) \in V(H)^{V(G)} \times  V(G) \Biggm|
                                             \sharp  \Supp f < \infty \right\},
    \end{gather*}
 where $\Supp f= \{ x \in V(G) \mid f(x) \neq 0 \}$ for a fixed element $0 \in V(H)$.  For $(f ,u ), (g,v) \in V( H \wr G )$, $ \left( (f ,u ), (g,v) \right) \in E ( H \wr G) $ 
  if either $(a)$ or $(b)$ is satisfied:  
    \begin{enumerate}
          \item[$(a)$]   $f=g \quad \text{and} \quad (u,v) \in E(G)  $,      
 
          \item[$(b)$]  $  u=v,   \quad f(x) = g(x)  \quad (\forall x \in V(G)\setminus \{ u \})
                                             \quad \text{and} \quad (f(u) , g(u)) \in E(H) $.
    \end{enumerate}
We call $G$ the underlying graph of $H \wr G$ and $H$ the fiber graph of $H \wr G$. 
\end{definition}     
Throughout the paper, we will only consider the case $H=\mathbb{Z}_2$ that consists of two vertices $\{0,1\}$, say, and one edge that connects the two vertices. 
(As noted in Remark \ref{Thm:rem25}(2), 
the results in this paper hold when $H$ is a finite graph.)  
We denote the elements of $V( \mathbb{Z}_2 \wr G)$ by bold alphabets $\bm{x} , \bm{y} \cdots $
 and the elements of $V(G)$ by standard alphabets $x,y, \cdots$.  \\

Next, we introduce the notion of weighted graphs.
Let $\mu: V(G) \times V(G) \to [0,\infty )$ be a symmetric function such that 
$\mu_{xy} = \mu (x,y)>0$ if and only if $(x,y) \in E(G)$.  
We call the pair $(G, \mu )$ a weighted graph.
 For a weighted graph $(G, \mu )$, we define a measure $m = m_G$ on $V(G)$ by $m(A) = \sum_{x \in A} m(x)$ for $A\subset V(G)$ where $m(x) = \sum_{y:y \sim x} \mu_{xy}$.
 We will write $ V(x,r) = V_G (x,r ) = m (B(x,r))$, where $B(x,r) = \{ y \in V(G) \mid d (x,y) \le r \}$.

Let $\{ X_n \}_{n \ge 0 }$ be the (time-homogeneous) random walk on $G$
 whose transition probability is $P = (p(x,y) )_{x,y \in V(G)}$, where 
 $p(x,y)=\mu_{xy}/m(x)$. 
We call $\{ X_n \}_{n \ge 0 }$ the random walk associated with the weighted graph $(G , \mu )$.
$\{ X_n \}_{n \ge 0 }$ is reversible w.r.t. $m$, i.e.
$m (x) p(x,y) = m (y) p(y,x)$ for all $x,y \in V(G)$. Define
\[
p_n(x,y):=P_x(X_n=y),\qquad  ~~\forall x,y \in V(G).
\]
$p_n(x,y)/m(y)$ is called the heat kernel of the random walk. 

\medskip

We next give a set of 
conditions for the graph and the random walks.
\begin{Assumption}   \label{Ass:rw}
 Let $(G,\mu )$ be a weighted graph. We consider the following assumptions for $(G, \mu )$.
   \begin{enumerate}
        \item[$(1)$] ($p_0$-condition) : $(G,\mu)$ satisfies $p_0$-condition, i.e. there exists $p_0 >0$ such that $\mu_{xy}/m(x)  \ge p_0$ 
              for all $x,y \in V(G)$ with $\{ x,y \} \in E(G)$.    

        \item[$(2)$] ($d_f$-set condition) : There exist positive constants $d_f, c_1 , c_2 >0$ such that 
                         \begin{gather}
                                                 c_1 r^{d_f}  \le V (x,r) \le c_2 r^{d_f}       \label{eq:dfset} 
                         \end{gather}    
            for all $x \in V(G) , r \in \mathbb{N} \cup \{ 0 \} $.          
               Here, we regard $0^{d_f}$ as $1$.  

     \item[$(3)$]  (On-diagonal heat kernel upper bound) : 
      There exists $d_s > 0$ such that 
              the heat kernel  $\displaystyle \frac{p_n (x,y)}{ m(y) } $ of $\{ X_n \}_{ n \ge 0  }$ satisfies  the following estimate for all $x,y \in V(G), n \ge 1$ :
                       \begin{align}
                                    \frac{p_n (x,y)}{m(y)}  
                                          &\le   c_1n^{-d_s/2}.  \label{eq:21dow} 
                       \end{align} 

       \item[$(4)$]  (Sub-Gaussian heat kernel estimates) :  
             There exists $d_w > 1$ such that 
              the heat kernel  $\displaystyle \frac{p_n (x,y)}{ m(y) } $ of $\{ X_n \}_{ n \ge 0  }$ satisfies  the following estimates:
                       \begin{align}  \label{eq:210} 
                                    \frac{p_n (x,y)}{m(y)}  
                                          &\le   \frac{c_1}{V(x,n^{\frac{1}{d_w}})} 
                                            \exp \left( -c_2 \left( \frac{d(x,y)^{d_w}}{n} \right)^{\frac{1}{d_w -1}} \right)  
                       \end{align} 
                            for all $x,y \in V(G), n \ge 1 $, and 
                       \begin{align}  \label{eq:211}
                                   \frac{p_n (x,y)}{ m(y) }  + \frac{ p_{n+1} (x,y) }{ m (y) }
                                               &\ge    \frac{c_3 }{V(x,n^{\frac{1}{d_w}})} 
                                               \exp \left( -c_4 \left( \frac{d(x,y)^{d_w}}{n} \right)^{\frac{1}{d_w -1}} \right)   
                       \end{align}   
                           for $x,y \in V(G), n \ge 1$ with $d(x,y) \le n$. 
   \end{enumerate}
\end{Assumption} 

We will clarify which of the conditions above are assumed in each statement. 
As one can easily see, Assumption \ref{Ass:rw} $(2)$ and \eqref{eq:210} imply Assumption \ref{Ass:rw} $(3)$ with
                          \begin{gather*}
                                       d_s/2 = d_f/d_w.
                         \end{gather*}
$d_s$ is called the spectral dimension.

The fractal graphs such as the Sierpinski gasket graph and the Sierpinski carpet graph given  in section \ref{Intro} satisfy Assumption \ref{Ass:rw} (see \cite{BB,Jones}).  
Note that under Assumption \ref{Ass:rw} $(1)$, $G$ satisfies \eqref{eq:bounded} with $M = 1/p_0$. 
  Also note that, under Assumption \ref{Ass:rw} $(2)$, we have $c_1 \le m_G (x)   \le c_2$ for all $x \in V(G)$ and hence   
    \begin{gather}
                  c_1 \sharp A  \le    m (A)  \le    c_2 \sharp A,      \qquad~~\forall A \subset V(G),  \label{eq:counting1} 
    \end{gather}  
where $\sharp A$ is the cardinal number of $A$. 
Finally, under Assumption \ref{Ass:rw} (1) (2) we have  
   \begin{gather} 
           0 < \inf_{x,y\in V(G) , x \sim y} \mu_{xy} \le \sup_{x,y \in V(G) , x \sim y}  \mu_{xy}  < \infty.     \label{eq:conduct1}
    \end{gather}

Next, we define the lamplighter walk on $\mathbb{Z}_2 \wr G$.
We denote the transition probability on $\mathbb{Z}_2$ by $P^{ ( \mathbb{Z}_2) } = ( p^{ ( \mathbb{Z}_2) } (a,b))_{a,b\in \mathbb{Z}_2} $,
 where $ P^{ ( \mathbb{Z}_2) } $ is given by
     \begin{gather*}
                  p^{ ( \mathbb{Z}_2) } (a,b) = \frac{1}{2}, \qquad \text{ for all $a,b \in \mathbb{Z}_2$. }    
      \end{gather*}
We can lift $P=(p(x,y))_{x,y\in G}$ and $P^{ ( \mathbb{Z}_2) } = ( p^{ ( \mathbb{Z}_2) } (a,b))_{a,b\in \mathbb{Z}_2} $ on $\mathbb{Z}_2 \wr G$, by
      \begin{align*}
               \tilde{p}^{ (G) }  ( (f,x) , (g,y) ) &=   
                       \begin{cases}
                                     p (x,y)    &   \text{if $ f=g$}    \\
                                       0            & \text{otherwise}
                       \end{cases},     \\
                \tilde{p}^{ (\mathbb{Z}_2) }  ( (f,x) , (g,y) ) &=   
                       \begin{cases}
                                     \frac{1}{2}   &    \text{if $x=y$ and $f(v) = g(v)$ for all $v \neq x$ }    \\
                                       0            & \text{otherwise}
                       \end{cases} .     \\           
      \end{align*}

Let $Y_n = \{ (\eta_n , X_n ) \}_{n \ge 0 }$ be a random walk on $\mathbb{Z}_2 \wr G$ 
 whose transition probability $\tilde{p}$ is given by 
               \begin{align*} 
             & ~~~~~~~  \tilde{p} ( (f,x) , (g,y) )  
                                        =  \tilde{p}^{ (\mathbb{Z}_2) }  \ast  \tilde{p}^{ (G) }  \ast  \tilde{p}^{ (\mathbb{Z}_2) }  ( (f,x) , (g,y) )  \\
                      =  & \sum_{ (h_1 , w_1) , (h_2 , w_2)}  
                            \tilde{p}^{ (\mathbb{Z}_2) } (  (f,x) , (h_1, w_1) )  \tilde{p}^{ (G) }  (  (h_1, w_1) ,  (h_2, w_2)   )   \tilde{p}^{ (\mathbb{Z}_2) }  ( (h_2,w_2) , (g,y) ) .
              \end{align*}

Note that 
 if $(f,x) , (g,y) \in V(\mathbb{Z}_2 \wr G)$ satisfy $x \sim y$ and $f(z) = g(z)$ for all $z \neq x,y$ then  
       \begin{gather*}
                  P(Y_{n+1} = (g,y) \mid Y_n = (f,x)  )  = \frac{1}{4} p(x,y),
      \end{gather*}
 and otherwise it is zero. 
 
 This random walk moves in the following way. 
 Let $X_n$ be the site of lamplighter at time $n$ and $\eta_n$ be the on-lamp state at time $n$.
 The lamplighter changes the lamp at $X_n$ with probability $1/2$, moves on $G$ according to the transition probability $P=(p(x,y))_{x,y \in G}$, and then 
 changes the lamp at $X_{n+1}$ with probability $1/2$. The lamplighter repeats this procedure. 
 (In the first paragraph of section \ref{Intro}, we discussed the case when $\{X_n\}$ is a simple random walk on $G$.) 

Note that $\{ Y_n \}_{n \ge 0 }$ is reversible w.r.t. $m_{\mathbb{Z}_2 \wr G}$,
 where 
  \begin{gather*}
              m_{\mathbb{Z}_2 \wr G} ( (\eta , x) ) = m(x).  
  \end{gather*}    
We denote the transition probability of $\{ Y_n \}_{n \ge 0 }$ as $p( \bm{x} , \bm{y} )$ 
(cf. $p(x,y)$ is the transition probability of $\{ X_n \}_{n \ge 0}$). 
We sometimes write $m (\bm{x} )$ instead of $ m_{\mathbb{Z}_2 \wr G} (\bm{x} )$.

\subsection{Main results}

In this subsection, we state the main results of this paper.

\begin{theorem}   \label{Thm:HK}
Suppose that Assumption $\ref{Ass:rw}$ (1), (2) and (4) hold. 
Then the following holds;
     \begin{gather*}
                      \frac{ p_{2n}   (\bm{x} , \bm{x} )} {m_{\mathbb{Z}_2 \wr G} (\bm{x} )}  \approx   \exp [ - n^{\frac{d_f}{d_f + d_w}} ],\qquad~~
                            \forall \bm{x} \in V( \mathbb{Z}_2 \wr G). 
      \end{gather*}  
  \end{theorem}

Next we state the results 
on the LILs when $d_s/2 < 1$ and $d_s/2 > 1$ respectively.
\begin{theorem}   \label{Thm:LILrectra}    
  Let $(G, \mu)$ be a weighted graph. 
  
 \noindent  (I) Assume that Assumption  \ref{Ass:rw} (1), (2), (4) and $d_s / 2 <1$ hold.
 Then there exist (non-random) constants $c_1 , c_2, c_3 , c_4 >0$ such that the following hold
  for all $\bm{x} \in V(\mathbb{Z}_2 \wr G)$: 
      \begin{align}
                c_1 &\le  \limsup_{n \to \infty} \frac{ d_{\mathbb{Z}_2 \wr G } (Y_0, Y_n) }{ n^{d_s/2}  (\log \log n)^{1-d_s/2} } \le c_2, \quad 
                                 \text{ $P_{ \bm{x} }$-a.s.}  \label{eq:LILrec10}  \\   
                c_3 &\le  \liminf_{n \to \infty}  \frac{  d_{\mathbb{Z}_2 \wr G } (Y_0, Y_n)  }{ n^{d_s/2} (\log \log n)^{-d_s/2} }  \le c_4, \quad 
                                 \text{ $P_{ \bm{x} }$-a.s.}   \label{eq:LILrec20}  
      \end{align}
\noindent (II)  Assume that Assumption  \ref{Ass:rw} (1), (2), (3) and $d_s / 2 >1$ hold. 
Then there exist (non-random) positive constants 
$c_1,c_2 >0$ such that the following hold for all $\bm{x} \in V(\mathbb{Z}_2 \wr G)$:
     \begin{align}
                 c_1 & \le  \liminf_{n \to \infty} \frac{d_{\mathbb{Z}_2 \wr G} (Y_0 , Y_n ) }{n}  
                        \le   \limsup_{n \to \infty}  \frac{d_{\mathbb{Z}_2 \wr G} (Y_0 , Y_n ) }{n}   \le c_2, \qquad 
                             \text{ $P_{ \bm{x} }$-a.s. }    \label{eq:LILtransi30}  
     \end{align}
\end{theorem}

\begin{remark} \label{Thm:rem25}
\begin{enumerate}
    \item[$(1)$] Note that  in Theorem \ref{Thm:LILrectra}(II) we do not need Assumption \ref{Ass:rw}(4) but only need the on-diagonal upper bound \eqref{eq:21dow}.  
           Since the transient case is discussed under a general framework in \cite{Okamura} (see subsection \ref{subsec:LIL44}),
           we do not pursue the minimum assumption for \eqref{eq:LILtransi30} to hold. 
    
    \item[$(2)$]  We can obtain the same results (by the same proof) if we replace $\mathbb{Z}_2$ by  
            a finite graph $H$ with $\sharp H \ge 2$ and  $p^{( \mathbb{Z}_2) }$ by $p^{(H)}$, where $p^{(H)}$ is the transition probability on $H$ given by 
                      \begin{gather*} 
                                      p^{(H)} (a,b) = \frac{1}{ \sharp H }, \qquad \text{for all $a,b \in V(H)$}.
                      \end{gather*}            

    \item[$(3)$]  For each $0 < \alpha < 1$, Rau \cite[Proposition 1.2]{Rau}  constructed a graph $G_{\alpha}$ such that
           the random walk on $G_{\alpha}$ satisfies the following heat kernel estimates:
                   \begin{gather} \label{HKvaro} 
                              p_{2n} (x,x) \approx  \exp ( -n^{\alpha} ).
                   \end{gather}
         For the case ${1}/{3} \le \alpha < 1$, the  graphs constructed by Rau are the wreath product on $\mathbb{Z}$, 
           but the fiber graphs are different site by site.
         (The definition of wreath product given by Rau is more general than ours.) 

          On the other hand,  for each $d_f , d_w $ such that $2 \le d_w \le 1 + d_f$ and $d_f \ge 1$, 
          Barlow \cite[Theorem 2]{Barlow2} constructed weighted graphs that satisfy Assumption  \ref{Ass:rw}. 
          Combining this and Theorem \ref{Thm:HK}, for any given $1/3 \le \alpha < 1$ we can give an alternative example where the heat kernel enjoys \eqref{HKvaro}. 
       
   \item[$(4)$]  For the case of $d_s/2 = 1$, we have not been able to obtain the LIL in general.     
        However, one can obtain the LIL for the case of $\mathbb{Z}^2$ as follows. (Note that $d_s/2 =1$ in this case since $d_f = d_w = 2$.)    
      
        Define $R_n =\sharp \{ X_0 , \cdots , X_n \}$. 
        Dvoretzky and Erd\H{o}s \cite[Theorem 1,4]{DE} proved the following law of large numbers of $R_n$:  
                     \begin{align*}
                               \lim_{n \to \infty} \frac{ R_n}{ \pi n/\log n}
                                        &= 1,   \qquad \text{$P$-a.s.}                                              
                     \end{align*}
        In Propositions \ref{Prop:LIL40} and \ref{Prop:LIL60}, we will show that $\frac{1}{4} R_n  \le d_{\mathbb{Z}_2 \wr \mathbb{Z}^2} (Y_0 , Y_n) $ 
        for all but finitely many $n$ and that there exists a positive constant $C>0$ such that $d_{\mathbb{Z}_2 \wr \mathbb{Z}^2} (Y_0, Y_n) \le CR_n$ for all $n$. 
       Using these facts, we see that there exist positive constants $c_1, c_2 >0$ such that for all $\bm{x} \in V(\mathbb{Z}_2 \wr G)$ 
                     \begin{align*}
                           c_1 \le  \liminf_{n \to \infty} \frac{ d_{\mathbb{Z}_2 \wr \mathbb{Z}^2} (Y_0 , Y_n)}{ n/\log n }
                                   \le \limsup_{n \to \infty} \frac{ d_{\mathbb{Z}_2 \wr \mathbb{Z}^2} (Y_0 , Y_n)}{ n/\log n} \le c_2,  \qquad \text{$P_{\bm{x}}$-a.s. }  
                     \end{align*}    
       As we see, the exponents differ from those in the cases of $d_s/2<1$ and $d_s/2 > 1$. \end{enumerate}
\end{remark}

\section{Consequences of heat kernel estimates}   \label{Sec:ConseqHK}
In this section, we give preliminary results obtained from the  sub-Gaussian heat kernel estimates \eqref{eq:21dow}, \eqref{eq:210} and \eqref{eq:211}. 
Throughout this section, we assume that Assumption \ref{Ass:rw} (1), (2) hold.

First, the following can be obtained by a simple modification of \cite[Lemma 3.9]{Barlow1}. 
 (Note that \eqref{eq:211} is not needed here.)  
\begin{lemma}  \label{Lem:lower40}
Suppose \eqref{eq:210}. Then there exist positive constants $c_{1} , c_{2} > 0$ such that 
                 \begin{gather*}
                         P_y \left( \max_{0 \le j \le n} d(x,X_j) \ge 3r \right)
                             \le  c_{1} \exp \left( -c_{2} \left( \frac{r^{d_w}}{n} \right)^{\frac{1}{d_w -1}} \right) 
                 \end{gather*}  
  for all $n \ge 1, r  \ge 1, x,y \in V(G)$ with $ d(x,y) \le r$.
\end{lemma}

\medskip

The following lemma will be used in subsection $\ref{subsec:LIL41}$. 
Note that unlike Lemma \ref{Lem:lower40}, we need only weaker condition \eqref{eq:21dow} here. 
\begin{lemma}   \label{Lem:LIL50}
        Suppose \eqref{eq:21dow}. Then there exist positive constants $c_1, c_2>0$ such that 
             \begin{gather}    \label{eq:LIL42}
                          P_x \left( \max_{0 \le j \le n } d(x,X_j) \le r \right)  \le c_1\exp \left( -c_2 \frac{n}{r^{d_w}} \right)   
            \end{gather} 
       for all $x \in V(G) , n,r \ge 1$.
\end{lemma}
\begin{proof}
We first show that there exists a positive constant $c_{1} > 0$ and a positive integer $R$ such that 
        \begin{gather}   \label{eq:LIL40}
               P_x \left( \max_{0 \le j \le [r^{d_w}] } d(x,X_j)  \le 2c_{1} r \right) \le \frac{1}{2}  
        \end{gather}
   for all $x \in V(G)$ and for all $r \ge R$, where $[r^{d_w}]$ means the greatest integer which is less than or equal to $r^{d_w}$.
Using \eqref{eq:21dow}, for $r\ge 1$ we have 
        \begin{align*}
               & P_x  \left(  \max_{0 \le j \le [r^{d_w}] } d(x,X_j)  \le 2 c_{1} r \right)    
                      \le   P_x (d(x,X_{ [r^{d_w}] }) \le 2c_{1}r )  \\
               & \le  c_2 \frac{1}{ ([r^{d_w}])^{d_f/d_w} }  \sum_{y \in B(x , 2c_{1}r) } m(y)  
                    \le  c_3   \frac{1}{ r^{d_f}} m( B(x,  2c_{1} r) ).  
        \end{align*}
Recall that $c_4 r^{d_f} \le m (B(x,r)) \le c_5 r^{d_f}$ for $r \in \mathbb{N} \cup \{  0 \} $ by \eqref{eq:dfset}. 
Take $c_1$ as a positive constant which satisfies $c_3 (2c_1)^{d_f} \le \frac{1}{2c_5}$.  
We first consider the case of $2c_1 r \ge 1$. In this case,  by the above estimate we have 
       \begin{align*}
                P_x  \left(  \max_{0 \le j \le [r^{d_w}] } d(x,X_j)  \le 2 c_{1} r \right)  
                      \le  c_3   \frac{1}{r^{d_f}} m( B(x,  2c_{1} r) )  
                      \le c_3  c_5 (2c_1)^{d_f} \le \frac{1}{2}.
       \end{align*}
 Next we consider the case of $2c_1 r < 1$. In this case, $m( B(x,2c_1 r) ) =m ( \{ x \} ) \le c_5$.  
 Take $R$ such that $R > (2c_3c_5)^{1/d_f}$.
 By the above estimate we have  
        \begin{align*}
               P_x \left( \max_{0 \le j \le [r^{d_w}] } d(x,X_j)  \le 2 c_{1} r \right)   
                       \le  c_3   \frac{1}{r^{d_f}} m( B(x,  2c_{1} r) ) 
                       \le  c_3   \frac{c_5}{r^{d_f}} \le \frac{1}{2}
        \end{align*}
   for $r \ge R$. We thus obtain \eqref{eq:LIL40}.

We now prove \eqref{eq:LIL42}. Let $r \ge R$. It is enough to consider the case $n\ge [r^{d_w}]$ since otherwise \eqref{eq:LIL42} is immediate 
by adjusting the constants. Let $k\ge1$ be such that $k[r^{d_w}]\le n< (k+1)[r^{d_w}]$
and let $t_i = i [r^{d_w}] $.
Then by the Markov property and (\ref{eq:LIL40}) we have
      \begin{align*}
                P_x \left( \max_{0 \le j \le n}   d(x,X_j ) \le   c_{1}r \right)   
                                      & \le P_x \left( \bigcap_{ 0 \le i \le k-1} \left\{ \max_{t_i \le j \le t_{i+1} }  d( X_{t_i} , X_{j } ) 
                                                                                   \le 2c_{1}r \right\} \right)   \\
                                      & \le  \left\{ \sup_y P_y \left(  \max_{0 \le j \le [r^{d_w}] }  d(y,X_j ) \le 2c_{1} r  \right) \right\}^k   \\
                                      & \le \left(  \frac{1}{2}   \right)^k  
                                             \le c_6 \exp (-c_7 k ) 
                                             \le c_8 \exp (-c_{9} nr^{-d_w}).
     \end{align*} 
It is immediate for the case of $1 \le r \le R$ from the above estimate. 
Hence we obtain  $(\ref{eq:LIL42} )$ by adjusting the constants.  
\end{proof}

\medskip 

In the next proposition, we show that Lemma \ref{Lem:LIL50} is sharp up to constants if we assume both $(\ref{eq:210})$ and $(\ref{eq:211})$. 
The idea of the proof is based on \cite[Lemma 7.4.3]{SC1}, where a similar estimate was given for a class of random walks on $\mathbb{Z}^d$. 

\begin{proposition}    \label{Prop:lower80}
Suppose \eqref{eq:210} and \eqref{eq:211}. 
 Then there exist $N_0 \ge 1$ and positive constants $ c_1, c_2>0$ such that 
      \begin{gather*}
             P_x \left( \max_{0 \le j \le n } d(x,X_j) \le r \right)  \ge c_1\exp \left( -c_2 \frac{n}{r^{d_w}} \right) 
      \end{gather*} 
  for all $r \ge N_0$ and $n \ge 4^{ \frac{d_w}{d_w -1} }$. 
\end{proposition}

The proof consists of the following two lemmas. 
\begin{lemma}  \label{Lem:lower50}
Under \eqref{eq:211} there exists $\epsilon \in (0,1)$  such that 
     \begin{gather*}
              P_y (d(x, X_{n} ) >   r  )    \le 1- \epsilon 
     \end{gather*} 
  for all $r\ge 1$, $n \ge 4^{\frac{d_w}{d_w -1}}$ with $n < r^{d_w}$ and $x, y  \in V(G) $ with $d(x,y) \le  r$.

\end{lemma}

\begin{proof}
We follow the argument in \cite[Lemma 7.4.7]{SC1}.  
Let $\gamma = 1/(d_w - 1)$.
Let $\ell_{x,y}$ be a geodesic path from $x$ to $y$ in $G$. Let $x_n \in \ell_{x,y}$ be the $([n^{1 / d_w} ] +1)$-th vertex from $y$. 
Then $B(x_n, [n^{1/d_w} ] ) \subset B(x,r-1)$ 
    since for all $ z \in B(x_n , [n^{1/d_w} ] ) $ we have $d(x,z) \le d(x,x_n) + d(x_n , z ) \le (d(x,y) - [n^{1 / d_w} ]-1) + [n^{1 / d_w} ] \le r-1$.
Also for all $z \in B(x_n ,  [n^{1/d_w} ] )$, we have $d(y,z) \le d(y , x_n) + d(x_n ,z)  \le 2  [n^{1 / d_w} ] +1$.
Hence, by \eqref{eq:211} and \eqref{eq:dfset} we have 
       \begin{align*}
                  & 2P_y \left( d(x,X_{n}) \le r \right) \ge  P_y \left( d(x , X_{n-1}) \le  r-1 \right) + P_y \left( d(x,X_{n}) \le r-1 \right)  \\    
                  \ge & c_1 \sum_{ z \in B(x_n , [n^{1/d_w}] )  } 
                              \frac{ m( z ) }{ V(y ,(n-1)^{1 / d_w}) } \exp \left[ -c_2 \left( \frac{d( y ,z )^{d_w}}{n-1} \right)^{ \gamma } \right]   \\
                  \ge & c_3 \left( \sum_{ z \in  B(x_n , [n^{1/d_w}] )  } \frac{ m(z) }{ n^{d_f / d_w} }  \right)  
                              \exp \left[ -c_2 (2 \cdot 3^{d_w})^{ \gamma }  \right] \ge  c_4 \exp \left[ -c_2  (2\cdot 3^{d_w})^{\gamma }  \right] ,
       \end{align*}
  provided $n \ge 4^{\frac{d_w}{d_w -1}}$ so that $d(y,z) \le 2n^{1/d_w} +1 \le 3 n^{1/d_w} \le  n-1$ for any $z \in B(x_n , [n^{1/d_w}])$.  
The proof completes by taking $\epsilon =   \frac{1}{2}  c_4 \exp \left[ -c_2  (2\cdot 3^{d_w})^{\gamma }  \right] $
(note that we may take $c_4<1$).  
\end{proof}

\begin{lemma} \label{Lem:lower70}
 Let $\epsilon$ be as in Lemma \ref{Lem:lower50}.  
 Then under \eqref{eq:210} and \eqref{eq:211} 
  there exists $\eta \ge 1$ such that for all $x,y \in V(G) $, positive integers $r \ge 8^{\frac{1}{d_w -1} }$ with $d(x,y) \le 2r$ 
     and for all integers $k \ge 0$ and  $\ell \ge 4^{\frac{d_w}{d_w -1}}$ with $k [r^{d_w}] \le \ell \le (k+1) [r^{d_w}]$, we have 
           \begin{gather*}
                  P_y \left( \max_{0 \le j \le \ell } d(x , X_j) \le 3 \eta  r,  d(x,X_{\ell} ) \le 2r \right) 
                               \ge \left( \frac{ \epsilon}{2} \right)^{k} \wedge \frac{\epsilon}{2}.
           \end{gather*}
\end{lemma}

\begin{proof}
We follow the argument in the proof of \cite[Lemma 7.4.3]{SC1}.
We prove the assertion by induction in $k$.

\underline{Step I}: We first prove the case $k=0,1$.  
Let $\gamma = 1/(d_w - 1)$.
  In general, $1 \le P(A) + P(B) + P( (A \cup B)^c )$ holds for any events $A,B$.   
  As $A, B$ take $A = \{ \max_{0 \le j \le \ell} d(x,X_j) > 3 \eta  r \}$, $B =\{ d(x,X_{\ell} ) > 2r \}$.
 Let $4^{\frac{d_w}{d_w -1}} \le \ell \le 2[r^{d_w}]$. By Lemma \ref{Lem:lower40} and Lemma \ref{Lem:lower50} we have
           \begin{align*}
                     1 &\le  P_y \left(  \max_{ 0 \le j \le \ell } d(x,X_j) > 3 \eta  r \right)   + P_y \left(   d(x,X_{\ell} ) > 2r \right)      
                                    + P_y \left(   \max_{0 \le j \le \ell} d(x , X_j)  \le  3 \eta  r , d(x,X_{\ell}) \le 2r \right) \\
                      &\le c_1 \exp \left[ - c_2 \left( \frac{ (\eta  r)^{d_w} }{\ell} \right)^{ \gamma } \right] 
                                + (1-\epsilon )  
                                +  P_y \left(   \max_{0 \le j \le \ell} d(x , X_j) \le  3 \eta  r , d(x,X_{\ell}) \le  2r \right) . 
           \end{align*}
 From above and using $\ell \le 2[r^{d_w}]$ we have 
          \begin{align*}
                      P_y \left(   \max_{0 \le j \le \ell} d(x , X_j) \le 3\eta  r , d(x,X_{\ell}) \le 2r \right) 
                      \ge \epsilon -  c_1 \exp \left[ - c_2 \left( \frac{ (\eta  r)^{d_w} }{\ell } \right)^{\gamma} \right]
                      \ge \epsilon - c_1 \exp \left[ -c_2  \left( \frac{ \eta^{d_w} }{2} \right)^{\gamma} \right].
           \end{align*}
Taking $\displaystyle \eta > \left\{ \frac{2^{\gamma}}{c_2} \log \left( \frac{2 c_1}{ \epsilon} \right) \right\}^{1/(\gamma d_w) } \vee 1$,  
we obtain
       \begin{gather*}
                     P_y \left( \max_{0 \le j \le \ell } d(x , X_j) \le 3\eta  r , d(x,X_{\ell}) \le 2r \right)  \ge \frac{\epsilon}{2}  
       \end{gather*}
 for $4^{\frac{d_w}{d_w -1}} \le \ell \le 2[r^{d_w}]$.

\underline{Step II}: Let $k \ge 1$ and assume that the result holds up to $k$.  
  Let $\ell $ satisfy  $(k+1) [r^{d_w}] \le \ell \le (k+2) [r^{d_w}]$. 
  Define $\ell^{\prime} = k[r^{d_w} ]$. 
  Then since $\ell^{\prime} \wedge (\ell - \ell^{\prime}) \ge [r^{d_w}] \ge \frac{1}{2} r^{d_w} \ge 4^{ \frac{ d_w}{ d_w -1}}$ by $r \ge 8^{\frac{1}{d_w -1}}$, 
       using the Markov property and induction hypothesis we have  
                 \begin{align*}
                          & P_y \left( \max_{ 0 \le j \le \ell } d(x,X_j) \le 3 \eta  r  ,  d(x,X_{\ell} ) \le 2r \right)  \\
                         \ge & P_y \left( \max_{ 0 \le j \le \ell }  d(x,X_j)  \le  3 \eta r , d(x,X_{\ell} ) \le 2r,  d(x,X_{\ell^{\prime}} ) \le 2r \right)   \\
                         = & E_y \left[ 1_{ \{ \max_{ 0 \le j \le \ell^{\prime} } d(x,X_j ) \le 3 \eta r  ,   d(x,X_{\ell^{\prime}} ) \le 2r \}  } 
                                  P_{ X_{\ell^{\prime}} }  \left( d(x, X_{\ell - \ell^{\prime} } ) \le 2r , 
                                                                       \max_{0 \le j \le \ell - \ell^{\prime} } d(x,X_j) \le 3 \eta  r \right)     \right] \\
                       \ge & \frac{\epsilon}{2}    P_y \left(   \max_{ 0 \le j \le \ell^{\prime} } d(x,X_j ) \le 3 \eta  r  ,  d(x,X_{\ell^{\prime}} ) \le 2 r \right)
                                 \ge \left( \frac{ \epsilon}{2} \right)^{k+1} .
                 \end{align*}
We thus complete the proof.              
\end{proof}

Given Lemma \ref{Lem:lower70}, it is straightforward to obtain 
Proposition \ref{Prop:lower80}.

\section{On diagonal heat kernel}  \label{Sec:HK}

In this section, we give the proof of Theorem \ref{Thm:HK}.

The lower bound follows by the same approach as in \cite[Section 7]{SC1} (cf. \cite[Section 7]{PS2} and \cite[Section 15.D]{Woess}). 
We use Proposition \ref{Prop:lower80} for the proof.

\medskip 

\begin{proof}[Proof of the lower bound of Theorem \ref{Thm:HK}.] 
 Let $\bm{x} = (\eta, x) \in V( \mathbb{Z}_2 \wr G)$.  
   As we said before we write $m_{\mathbb{Z}_2 \wr G} $ as $m$.
 For any finite subset $A \subset V(\mathbb{Z}_2 \wr G)$, using 
 the Cauchy-Schwarz inequality we have
       \begin{align}
               \frac{p_{2n} (\bm{x} ,\bm{x} )}{ m(\bm{x} ) }                       
                    &= \sum_{ \bm{y} \in V(\mathbb{Z}_2 \wr G) }  \frac{p_n ( \bm{x} , \bm{y}) p_{n} ( \bm{y} , \bm{x} )}{ m(\bm{x}) } 
                         = \sum_{ \bm{y} \in V(\mathbb{Z}_2 \wr G)  }  \frac{ p_n ( \bm{x} , \bm{y} )^2 }{ m(\bm{y} )  } 
                         \ge  \sum_{ \bm{y} \in A}  \frac{ p_n (\bm{x} , \bm{y} )^2 }{ m(\bm{y} ) }  
                         \ge   \frac{1}{ m(A) }  P_{\bm{x} } \left( Y_n \in A \right)^2.   \label{eq:ODU30}   
       \end{align}
Set $A:= \{ \bm{y} = (f , y) \in V(\mathbb{Z}_2 \wr G) \mid y \in B_G (x,r),  f(z) = 0$ for all $z \in V(G)$ such that $d(x,z) >r\}$. 
Using $(\ref{eq:dfset} )$ and $(\ref{eq:counting1})$, we have 
         \begin{align*}
                m_{\mathbb{Z}_2 \wr G} (A)   = \sum_{ y \in B_G (x,r) } m_G (y) 2^{ \sharp  B_G (x,r)}   
                   \le   c_1 r^{d_f} 2^{c_2 r^{d_f} },  
          \end{align*}
 and using Proposition \ref{Prop:lower80} we have
       \begin{align*}
                P_{\bm{x} } \left( Y_n \in A \right)   \ge P_x \left( \max_{0 \le j \le n} d(x, X_j)  \le r \right) 
                    \ge c_3 \exp \left[- c_4 \frac{n}{r^{d_w}} \right]
       \end{align*}
   provided $n \ge 4^{\frac{d_w}{d_w -1}}$ and $r \ge N_0$.  
Hence,  by $(\ref{eq:ODU30})$ we have
    \begin{align*}
              \frac{ p_{2n} (\bm{x} , \bm{x} )}{ m(\bm{x} ) } \ge c_5 \exp \left[ -c_6 \left( d_f \log r  +  r^{d_f}  +  \frac{n}{r^{d_w}} \right) \right] .
    \end{align*}
 Optimize the right-hand side (take $r= n^{1/(d_f + d_w)}$), 
 then we obtain 
        \begin{gather*}
                     \frac{p_{2n} (\bm{x} , \bm{x} )}{ m( \bm{x} ) }  
                           \ge  c \exp \left[ -C  n^{\frac{d_f}{ d_f + d_w }} \right] 
        \end{gather*}
    provided $n \ge 4^{\frac{d_w}{d_w-1}} \vee N_0^{d_f + d_w}$. 
 The lower bound for $1 \le n \le 4^{\frac{d_w}{d_w-1}} \vee N_0^{d_f + d_w}$ is obvious, and we thus complete the proof.  

\end{proof}

\bigskip

We next prove the upper bound of Theorem \ref{Thm:HK} (cf. \cite[Section 8]{PS2} and \cite[Section 15.D]{Woess}).

\begin{proof}[Proof of the upper bound of Theorem \ref{Thm:HK}.] 
Without loss of generality, we may and do assume $\eta_0$ is identically $0$.  
For the switch-walk-switch random walk $\{ Y_n = (\eta_n , X_n ) \}_{n \ge 0}$ on $\mathbb{Z}_2 \wr G$, 
   $\eta_n$ is  equi-distributed on $\{ f \in \prod_{z \in V(G)}  \mathbb{Z}_2 \mid \Supp (f - \eta_0) \subset \bar{R}_n \} $, 
   where $\bar{R}_n = \{ X_0, X_1, \cdots, X_n \}$. 
   Hence, setting $R_n = \sharp \bar{R}_n$, we have 
\[
    P_{\bm{x}}  \left(  Y_n = \bm{x} \right)  =  \sum_{k=0}^{n+1}  P_{ \bm{x} } \left( Y_n = \bm{x}  , R_n = k \right)  
                                                        \le   \sum_{k=0}^{n+1}   E_{ \bm{x} }   \left[  1_{ \{ R_n = k \}  } 2^{-k}   \right] \le  E_{ \bm{x} }  \left[  2^{-R_n}   \right].    
\]  
In \cite[Theorem 1.2]{Gibson},  Gibson showed the following Donsker-Varadhan type range estimate: for any $\nu > 0 $ and any $x \in V(G) $,
\[
           - \log E_x \left[ \exp \left\{ - \nu m ( \bar{R}_{ [n^{d_w} V(x,n)] } ) \right\} \right]  \asymp  V(x,n).      \] 
Note that $V(x,n) \asymp n^{d_f}$. Replacing $n$ with  $n^{1/(d_f + d_w) }$ we have  
\[
            E_x \left[ \exp \left\{ -\nu m( \bar{R}_{n} ) \right\}  \right]  \approx \exp \left[ - n^{d_f / (d_f + d_w) }  \right] .   
\]
Since $cm(\bar{R}_n ) \le R_n$ (due to $(\ref{eq:counting1})$),  
by the above estimates, we obtain the upper estimate, i.e.
  \begin{gather*}
                   \frac{p_{n} (\bm{x} , \bm{x} )}{ m( \bm{x} ) }  \le  c \exp \left[ -C  n^{\frac{d_f}{ d_f + d_w }} \right]   .
  \end{gather*}
We thus complete the proof.  

\end{proof}

\section{Laws of the iterated logarithm}     \label{Sec:LIL}
 In this section, we will prove Theorem \ref{Thm:LILrectra}.

We first explain the idea of the proof. 
 Let $(G, \mu)$ be a weighted graph such that $G$ is of bounded degree. 
 For notational simplicity, let $o \in V(G)$ be a distinguished point and $\bm{0}  $ be the element of  $(\mathbb{Z}_2)^{V(G)}$  
   such that $\bm{0} (v) = 0$ for all $v \in V(G)$.  
In order to realize a given lamp state $(\eta , x) \in V(\mathbb{Z}_2 \wr G)$ from the initial lamp state $(\bm{0} , o) \in V(\mathbb{Z}_2 \wr G)$, 
  we need to visit all the ``on-lamp vertices''.  So
     \begin{align}
           & ~~~  \sum_{i \in V(G) }  \eta(i)  \le   d_{ \mathbb{Z}_2 \wr G } ( (\bm{0}, o)  , (\eta , x) )    \notag   \\ 
           &\le   \text{(the  minimum number of steps to visit all the ``on-lamp vertices'' from $o$ to $x$)}   +  \sum_{i \in V(G) }  \eta(i) .  \label{eq:distest10}
     \end{align}

 We apply the above observation to the lamplighter walk $\{ Y_n = (\eta_n , X_n) \}_{n \ge 0}$ on $\mathbb{Z}_2 \wr G$. 
 Note that the lamp at a certain vertex of $G$ (say $z$) cannot be changed without making the lamplighter visit at $z$. 
From this and (\ref{eq:distest10}), we see that $d_{ \mathbb{Z}_2 \wr G } (Y_0 , Y_n ) $ is heavily related 
 to the range of random walk $\{ X_n \}_{n \ge 0}$ on $G$.
Set $R_n = \sharp \{ X_0 , X_1 , \cdots , X_n \}$.
 Intuitively, $\sum_{i \in V(G) }  \eta_n (i) $ is close to $\frac{1}{2}  R_n$.
Indeed, we will show the following in Proposition \ref{Prop:LIL40} and Proposition \ref{Prop:LIL60}:      
       \begin{gather}
              \frac{1}{4}  R_n \le d_{ \mathbb{Z}_2 \wr G } ( Y_0 ,Y_n )  \quad \text{for all but finitely many $n$,} \quad a.s.,   \label{eq:LIL10} \\   
                    d_{ \mathbb{Z}_2 \wr G } ( Y_0 ,Y_n )    \le (2M+1)    R_n,      \qquad \text{for all $n$},     \label{eq:LIL20}
         \end{gather}
 where $M$ is defined by $(\ref{eq:bounded})$. 
We will prove $(\ref{eq:LIL10})$ and $(\ref{eq:LIL20})$ in subsection  \ref{subsec:LIL41}. 
The behavior of $R_n$ differs for $d_s /2< 1$ and $d_s /2 > 1$.
In subsection \ref{subsec:LILrec} (resp. \ref{subsec:LIL44}), 
 we prove the LILs of $R_n$ and $d_{ \mathbb{Z}_2 \wr G } (Y_0 , Y_n ) $ for $d_s /2 <1$ (resp. for $d_s/2 >1$).

\subsection{ Relations between the distance and the range}   \label{subsec:LIL41}
The main goal of this subsection is to prove $(\ref{eq:LIL10})$ and $(\ref{eq:LIL20})$.

\begin{proposition}   \label{Prop:LIL40} 
 Suppose that Assumption \ref{Ass:rw} (1), (2) and (3) hold. Then the following holds; 
           \begin{gather*}
                    \frac{1}{4} R_n  \le \sum_{i \in \{ X_0, X_1, \cdots, X_n \} } \eta_n (i) 
                           ~~ \text{for all but finitely many $n$}, ~~ \text{$P_{(\bm{0},x)}$-a.s., for all $x \in V(G)$.}
             \end{gather*}
\end{proposition} 

\begin{proof}
We fix $x \in V(G)$ and write $P$ instead of $P_{(\bm{0},x)}$. 
Define $S_n = \sum_{i \in \{ X_0, X_1, \cdots , X_n \}} \eta_n (i) $. 
It is easy to see that 
     \begin{gather*}
                    P (S_n = l \mid R_n = k)  = \left( \frac{1}{2} \right)^{k}
                        \begin{pmatrix}
                                 k \\
                                 l
                        \end{pmatrix} 
     \end{gather*}
 for $0 \le l \le k$. Then we have
      \begin{align}
            P \left( S_n \le \frac{1}{4} R_n \right) 
                  &= \sum_{l=0}^n P \left( S_n \le \frac{1}{4} l , R_n = l \right)  
                       = \sum_{l=0}^n \sum_{m=0}^{ \left[ \frac{1}{4} l \right]}   \left( \frac{1}{2} \right)^{l} 
                                      \begin{pmatrix}
                                             l \\
                                             m
                                       \end{pmatrix}    
                                  P(R_n = l) \nonumber\\
                     &\le  \sum_{l=0}^n \exp \left( - \frac{1}{16} l \right) P(R_n = l)   
                                \qquad \text{by the Chernoff bound}\nonumber\\
                     &\le P \left( R_n \le  n^{\frac{1}{2d_w}} \right) + 
                            \exp \left( -\frac{1}{16} n^{\frac{1}{2d_w}} \right) P(R_n \ge n^{\frac{1}{2d_w}} )  \nonumber \\
                      & \le P  \left(   \sup_{0 \le l \le  n  }  d(X_0 , X_{l} )  \le   n^{\frac{1}{2d_w}}  \right)   
                                 +   \exp \left( - \frac{1}{16} n^{\frac{1}{2d_w}} \right)   \nonumber\\
                      &\le c_1 \exp ( -c_2 n^{ \frac{1}{2}} )   + 
                                  \exp \left( -\frac{1}{16} n^{\frac{1}{2d_w}} \right) ,    
                                  \qquad \text{ by  Lemma $\ref{Lem:LIL50}$. }\nonumber
      \end{align}
Using the Borel-Cantelli lemma, we complete the proof.  
\end{proof}

 \begin{proposition}   \label{Prop:LIL60}  
  Let $(G, \mu)$ be a weighted graph such that $G$ is of bounded degree, and set $M = \sup_{v \in V(G) } \deg v (<\infty )$.
  Then for all realization of $\{ Y_n \}_{n =0}^{\infty}$ and all $n \ge 0$, 
       \begin{gather*}
                d_{ \mathbb{Z}_2 \wr G} (Y_0 , Y_n )  \le  (2M+1)  R_n.
       \end{gather*}
\end{proposition}

To prove the above proposition, we need the following lemma.
To exclude ambiguity, we first introduce some terminology.  
Let $H$ be a connected subgraph of $G$. 
   \begin{itemize}
          \item   A {\it path} 
                     $\gamma $ in $H$ is a sequence of vertices $v_0 v_1 \cdots v_k$ such that $v_i \in V(H)$ and   
                     $v_i v_{i+1} \in E(H)$  
                     for all $i $. 
                     For a path $\gamma$, we set $V(\gamma ) = \{ v_0 , v_1 , \cdots , v_k \}$,
                     and define the length of $\gamma $ as $ |\gamma | = k$. 
                     $e_j = v_j v_{j+1} ( j=0,1, \cdots , k-1)$ are said to be the {\it edges} of $\gamma$.  
  
          \item For the path $\gamma$ given above and a  
                    given edge $e \in E(G)$, we define $F(\gamma, e) = \{ l \in \{ 0,1,\cdots,k-1 \} \mid e_{l} = e \}$. 
   
          \item We denote $\overrightarrow{e} = \overrightarrow{uv}$ if the edge $e$ is directed from $u $ to $v$ and call  
                 $\overrightarrow{e}$  a directed edge.
                  For two directed edges $\overrightarrow{e_1} = \overrightarrow{u_1 v_1}$ and  $\overrightarrow{e_2} =  \overrightarrow{u_2 v_2}$, 
                    $\overrightarrow{e_1} $ and  $\overrightarrow{e_2 }   $ are equal if and only if
                    $u_1 = u_2, v_1 = v_2$.
    \end{itemize}

\begin{figure}[htb]
      \begin{minipage}{0.5\hsize}
                 \begin{center}
                            \includegraphics[width=50mm]{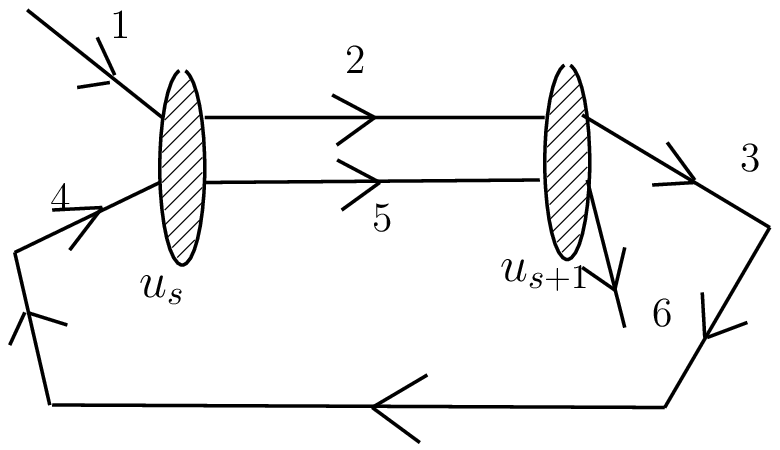}
                 \end{center}
                 \caption{an example of the path $\eta$. }
                 \label{fig:graph10}
     \end{minipage} 
     \begin{minipage}{0.5\hsize}
                \begin{center}
                            \includegraphics[width=50mm]{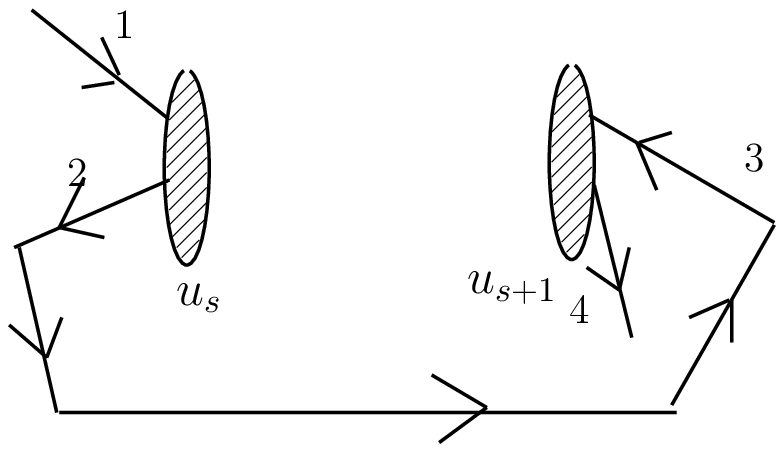}
                \end{center}
               \caption{an example of the surgery of $\eta$. }
               \label{fig:graph20}
     \end{minipage}
\end{figure}

\begin{lemma} \label{Lem:LIL70}
  Let $G$ be a graph of bounded degree, $M = \sup_{v \in V(G) } \deg v (<\infty )$ and $H$ be a finite connected subgraph of $G$.  
 
  \begin{enumerate}
      \item[$(1)$]  Let $x,y \in V(H)$. Then there exists a path $\gamma = w_0 w_1 \cdots w_k$ in $H$ such that the following hold. 
           \begin{enumerate}
                 \item[$(a)$] $\{ w_0, w_1, \cdots, w_k \} = V(H)$. 
                 \item[$(b)$]   Define $\tilde{e}_j = w_j w_{j+1}$ for $j=0,1, \cdots, k-1$. 
                                    Then $\sharp \left\{ j \in \{ 0,1,\cdots, k-1 \} \mid \tilde{e}_j = \tilde{e}_s \right\} \le 2$ for all $s=0,1,\cdots, k-1$.
                 \item[$(c)$] $w_0 = x$ and $w_k = y$. 
           \end{enumerate}

      \item[$(2)$] Let $\gamma$ be as in $(1)$. Then $|\gamma | \le 2 M \sharp V(H)$. 
  \end{enumerate}
\end{lemma}

\begin{proof}
 $(1)$ Take a path  $ \eta  = u_0 u_1 \cdots u_n $ in $H$ such that $u_0 = x$, $u_n = y$ and $\{ u_0 , u_1 , \cdots , u_n \} = V(H)$.
  Define $f_j = u_j  u_{j+1}$. 
  If each edge $f_j  $ satisfies  $\sharp \{ l \in \{ 0,1,\cdots , n-1 \} \mid f_l  =  f_j  \}  \le 2$ for $j=0,1, \cdots, n-1$,  
    then $\eta $ satisfies the conditions $(a), (b)$ and $(c)$.
 So we may assume that there exists $f_j$ such that $\sharp \{ l \in \{ 0,1,\cdots , n-1 \} \mid  f_l  =  f_j  \}  \ge 3$.  
  For such an edge $f_j$, there exist at least two elements $s,t \in F(\eta , f_j)$ such that $\overrightarrow{u_s u_{s+1}} = \overrightarrow{u_t u_{t+1}}$.
   Let $s<t$.
   Define $\eta_{st} = u_s u_{s+1} u_{s+2}  \cdots u_{t-1} u_t $  (see Figure \ref{fig:graph10}).
   Replace $\eta = u_0 \cdots u_{s-1} \eta_{st} u_{t+1} \cdots u_n $ by 
      $\tilde{ \eta } = \tilde{u_0} \tilde{u_1} \cdots \tilde{u_n} = u_0 \cdots u_{s-1} \tilde{ \eta }_{st} u_{t+1} \cdots u_n $ 
    where    $\tilde{\eta}_{st} = u_s u_{t-1} u_{t-2}  \cdots u_{s+2} $. 
    $\tilde{ \eta } $ is again a path, $V(\tilde{\eta} ) = V(H)$ and 
    $\sharp F( \tilde{ \eta }, f_j) =\sharp F(\eta , f_j) - 2$ (see Figure \ref{fig:graph20}).
    Repeat this operation to $f_0, f_1, \cdots , f_{n-1}$ inductively until obtaining 
    a path satisfying $(a),(b)$ and $(c)$.

$(2)$  Note that $w_j$ appears in $V(\gamma )$ at most $2 \deg (w_j)$ times for each vertex $w_j \in V(H)$ . 
The conclusion can be verified  easily.   
\end{proof}

\begin{remark}
While the revision of this paper was under review, Pierre Mathieu pointed us out a simpler proof of this lemma 
(therefore a simpler proof of Proposition $\ref{Prop:LIL60}$). Take a spanning tree of $V(H)$ and let
$\gamma'$ be an exploration path of the spanning tree from $x$ to $x$ 
(i.e. a path that crosses each bond of the tree exactly twice). Then one can produce a desired path $\gamma$ by an easy modification of $\gamma'$.  
\end{remark}

\begin{proof}[Proof of  Proposition $\ref{Prop:LIL60}$.] 
The graph $G^{\prime} = (V(G^{\prime}), E(G^{\prime}))$, where $V(G^{\prime}) = \{ X_0 , X_1 , \cdots , X_n \}$ and 
    $E(G^{\prime}) = \{  X_i X_{i+1}  \in E(G) \mid 0 \le i  \le n-1 \}$, is itself a connected subgraph of $G$. 
So  applying Lemma $\ref{Lem:LIL70}$ to $G^{\prime}$, we have 
   \begin{gather*}
        \min \{ | \gamma | \mid \gamma \text{ is a path starting at } X_0, V(\gamma ) = \{ X_0 , X_1 , \cdots , X_n \} \}  \le 2 M  R_n .
   \end{gather*}
By this and $(\ref{eq:distest10})$, we obtain
$d_{\mathbb{Z}_2 \wr G} (Y_0 , Y_n)  \le (2M + 1)  R_n$.   
\end{proof}

\subsection{Proof of Theorem \ref{Thm:LILrectra}(I)}  \label{subsec:LILrec}

In this subsection, we prove the LILs for $\{ Y_n \}_{n \ge 0}$ when $d_s/2 < 1 $.

\begin{theorem}     \label{Thm:LILdisc}
Assume that Assumption \ref{Ass:rw} (1), (2), (4) and $d_s/2 < 1$ hold.
Then there exist (non-random) constants $c_1 , c_2>0$ such that the following hold:
      \begin{align*}
                \limsup_{n \to \infty} \frac{R_n}{ n^{d_s/2} (\log \log n)^{1-d_s/2} } 
                          = c_1,   \qquad \text{$ P_x $-a.s. $\,\, \forall x \in V(G)$},       \\
                 \liminf_{n \to \infty} \frac{R_n}{ n^{d_s/2} (\log \log n)^{-d_s/2} } 
                          = c_2,   \qquad  \text{$P_x $-a.s. $\,\, \forall x \in V(G)$} .   
      \end{align*}
\end{theorem}

This is a discrete analog of \cite[Propositions 4.9 and 4.10]{BassKumagai10}. 
Note that the proof of these propositions relies on the self-similarity of the process.  
Since our random walk does not satisfy this property, we need non-trivial modifications for the proof.
Quite recently, Kim, Kumagai and Wang \cite[Theorem 4.14]{KKW} 
proved the LIL of the range for jump processes without using self-similarity of the process.  
By easy modifications, we can apply their argument to our random walk.   
The proof of Theorem \ref{Thm:LILdisc} will be given in Appendix \ref{sec:Appendix}.  

\medskip

\begin{proof}[Proof of Theorem  \ref{Thm:LILrectra}(I)] 
Note that Assumption \ref{Ass:rw} (1) implies that $G$ is of bounded degree. 
Thus by $(\ref{eq:distest10})$, Proposition  \ref{Prop:LIL40},  Proposition \ref{Prop:LIL60} and Theorem \ref{Thm:LILdisc}, we obtain \eqref{eq:LILrec10} and \eqref{eq:LILrec20}.   
\end{proof}

\subsection{Proof of Theorem \ref{Thm:LILrectra}(II)}   \label{subsec:LIL44}
In this subsection, we prove the LILs for $\{ Y_n \}_{n \ge 0}$ when $d_s/2 >1$.

First, we explain the notion of ``uniform condition'' defined in \cite{Okamura}.
We define the Dirichlet form $\mathcal{E}$ on  the weighted graph $(G, \mu )$ by  
     \begin{align*}  
                    &\mathcal{E} (f,f)  = \sum_{x,y \in V(G) }   ( f(x) - f(y) )^2 \mu_{xy},   
     \end{align*}
for $f: V(G) \to \mathbb{R}$, and the effective resistance $R_{ \text{eff} } (\cdot ,\cdot ) $ as
      \begin{align}  \label{eq:Res}
                    &R_{ \text{eff}} (A,B)^{-1}  = \inf \{ \mathcal{E} (f,f)  ; f|A = 1, f|B = 0 \}   
     \end{align}
 for $A,B \subset V(G)$ with $A \cap B = \emptyset$. 
Denote  $\rho (x,n) = R_{ \text{eff}} ( \{ x \} , B(x,n)^c )$ 
      for any $x \in V(G), n \in \mathbb{N}$ and  $\rho (x) = \lim_{n \to \infty} \rho (x,n)$.

\begin{definition}[Okamura \cite{Okamura}]
      We say that a weighted graph $(G,\mu )$ satisfies the uniform condition $(U)$ if $\rho (x,n)$ converges uniformly in $x \in V(G)$ to $\rho (x)$ as $n \to \infty$. 
\end{definition}
 For $A \subset G$, define
              \begin{align*}
                         T_A^{+} &= \inf \{ n \ge 1 \mid X_n \in A \} .
             \end{align*} 
  We write $T_x^+$ instead of  $T_{ \{ x \} }^+$.

The following is an improvement of \cite[Corollary 2.3]{Okamura}.
 \begin{proposition}   \label{Prop:LILTransi30} 
 Let $(G, \mu )$ be a weighted graph satisfying $(U)$ and \eqref{eq:conduct1} and assume that $G$ is of bounded degree.
 If $\sup_{x \in V(G)} P_x (M < T_x^{+} < \infty ) = O (M^{-\delta})$ for some $\delta >0$, then 
      \begin{gather*}
       1-F_2 \le \liminf_{n \to \infty}  \frac{R_n}{n}   \le \limsup_{n \to \infty}   \frac{R_n}{n}  \le  1 - F_1 ,   
                \qquad \text{$P_x$-a.s.}
      \end{gather*} 
 for all $x \in V(G)$, where $F_1 = \inf_{x \in V(G)} P_x ( T_x^{+} < \infty )$ and $F_2 = \sup_{x \in V(G)} P_x (T_x^+ < \infty )$.  
\end{proposition}

\begin{remark}
In \cite[Corollary 2.3]{Okamura},    a stronger condition $\sup_x P_x (M < T_x^+ < \infty ) = O(M^{-1-\delta } )$ for some $\delta >0$ is imposed to prove
                  $1-F_2 \le \liminf_{n \to \infty}  \frac{R_n}{n} $.
As we prove below, it is enough to assume $\sup_x P_x (M < T_x^+ < \infty ) = O(M^{-\delta } )$. 
\end{remark}

\begin{proof}[Proof of Proposition \ref{Prop:LILTransi30}.]
For the upper bound  $\limsup_{n \to \infty}   \frac{R_n}{n}  \le  1 - F_1$,
 \cite[Proof of Corollary 2.3]{Okamura} goes through without any modifications.

Hence we prove  $ 1-F_2 \le \liminf_{n \to \infty}  \frac{R_n}{n} $ under our assumption.    
  Fix $\epsilon >0$. By \cite[(2.5), (2.6), (2.7)]{Okamura} there exists $a \in (0,1)$ such that  for any $n$ and $M$
  we have 
                 \begin{gather}
                         P_x \left( \frac{R_n}{n} \le 1-F_2 - \epsilon \right) 
                                     \le \frac{2}{\epsilon} \sup_{x \in V(G)} P_x (M < T_x^+ < \infty )  + (M+1) a^{n/(M+1)}.  \label{eq:LILtransi20}  
                 \end{gather}

 Choose $k > 2/\delta$. Replacing $n $ by $n^k$ in $(\ref{eq:LILtransi20})$, 
 we have
            \begin{align*}
                             P_x \left( \frac{R_{n^k}}{n^k} \le 1-F_2 - \epsilon \right) 
                                     &\le \frac{2}{\epsilon}  \sup_{x \in V(G)}  P_x (M < T_x^+ < \infty )  + (M+1) a^{n^{k}/(M+1)}  \\ 
                                                                            &=\frac{2}{\epsilon}    O(M^{-\delta } )     + (M+1)a^{n^{k}/(M+1)}. 
            \end{align*}
 
Take $M=M (n) = n^{k/2} -1$ and we have 
                \begin{gather*}
                             P_x \left( \frac{R_{n^k}}{n^k} \le 1-F_2 - \epsilon \right)  \le \frac{2}{\epsilon}  O \left( \frac{1}{n^{k\delta /2} } \right)  +n^{k/2} a^{n^{k/2} }.  
                \end{gather*}
 Since  $k\delta /2 >1$, we can apply the Borel-Cantelli lemma and we obtain
               \begin{gather*}
                             1 - F_2 \le \liminf_{n \to \infty} \frac{R_{n^k}}{n^k},  \qquad \text{$P_x$-a.s.}             
               \end{gather*}
 For any $m$, choose $n $ as $n^k \le m < (n+1)^{k}$, and we then have  
               \begin{gather*} 
                           \frac{R_m}{m} \ge \frac{n^k}{m} \frac{R_{n^k}}{n^k} = \left( \frac{n}{n+1} \right)^k  \frac{(n+1)^k}{m} \frac{R_{n^k}}{n^k}  
                                       \ge   \left( \frac{n}{n+1} \right)^k  \frac{R_{n^k}}{n^k}.
               \end{gather*}
   Take $\liminf_{m \to \infty}$ and we obtain $1-F_2 \le \liminf_{n \to \infty} \frac{R_n}{n}$. 

\end{proof}

\begin{proof}[Proof of Theorem \ref{Thm:LILrectra}(II)]
 Note that the uniform condition $(U)$ is satisfied in our framework  by \cite[Proposition 4.6]{Okamura}.

Since $d_s/ 2 > 1$, we have
               \begin{align}  \label{eq:tail10}
                            P_x (M < T_x^+ < \infty )  \le    \sum_{n=M+1}^{\infty}  p_n (x,x)  \le  c\sum_{n=M+1}^{\infty}  n^{-d_s/2} =    O(M^{1-d_s/2}).  
               \end{align}
Note that Assumption \ref{Ass:rw} (1), (2) imply \eqref{eq:bounded} and \eqref{eq:conduct1}.   
By \eqref{eq:tail10} and Proposition $\ref{Prop:LILTransi30}$, we have 
               \begin{align}
                          1-F_2 \le \liminf_{n \to \infty}  \frac{R_n}{n}   \le \limsup_{n \to \infty}   \frac{R_n}{n}  \le  1 - F_1 ,   
                                       \qquad \text{$P_x$-a.s.}     \label{eq:LILtransi90}
               \end{align}

Define $G(x,x) = \sum_{n=0}^{\infty} p_n (x,x)$, and $F(x,x) = \sum_{n=1}^{\infty} P_x (T_x^{+} = n) = P_x ( T_x^{+} < \infty )$.
It is well known that 
   \begin{gather}
                      G(x,x) -1 = F(x,x) G(x,x).    \label{eq:functional}
   \end{gather}  
Since $d_s/2 >1$, we have
    \begin{gather*}
               \sup_{ x \in V(G) }  \sum_{n = 0}^{\infty}   p_n (x,x)  \le c \sum_{n=1}^{\infty} \frac{1}{n^{d_s / 2}}  < \infty. 
    \end{gather*}
By this and $(\ref{eq:functional})$ we have
     \begin{gather}
                      F_2 = \sup_x F(x,x) < 1.    \label{eq:LILtransi95}
     \end{gather}
Thus, by \eqref{eq:distest10}, Proposition \ref{Prop:LIL40}, Proposition \ref{Prop:LIL60}, \eqref{eq:LILtransi90} and \eqref{eq:LILtransi95}, 
 we conclude that for all $\bm{x} \in V(\mathbb{Z}_2 \wr G)$,
     \begin{align*}
               0<  \frac{1}{4} (1-F_2)  & \le \liminf_{n \to \infty} \frac{ d_{\mathbb{Z}_2 \wr G} (Y_0, Y_n) }{n}    \\
                                            & \le  \limsup_{n \to \infty} \frac{ d_{\mathbb{Z}_2 \wr G} (Y_0,Y_n )}{n} \le (2M+1) (1-F_1)< \infty , 
                  \qquad \text{ $P_{\bm{x}}$-a.s.}
     \end{align*}

Hence we complete the proof. 
\end{proof}

\begin{appendices}
 
\section{Proof of  Theorem \ref{Thm:LILdisc} }   \label{sec:Appendix}
\renewcommand{\theequation}{A.\arabic{equation}}
 \setcounter{equation}{0}
In this section, we will explain briefly the essential part of the proof of  Theorem \ref{Thm:LILdisc}, which is a discrete analog of \cite[Propositions 4.9 and 4.10]{BassKumagai10}. 
Note that the results in \cite{BassKumagai10} are for the range of Brownian motion on fractals, and 
the proof heavily relies on the self-similarity of Brownian motion.
Quite recently, Kim, Kumagai and Wang \cite[Theorem 4.16]{KKW} have obtained 
the LIL for the range of jump processes on metric measure spaces without using scaling law of the process. 
We employ the results and techniques in \cite{BassKumagai10}, \cite{Croydon} and \cite{KKW}, 
and prove the LIL for the range of the random walk without using scaling law of the process or of the heat kernel. 

The key to prove the LILs for the range of the process is to establish those for the maximum of local times. 
We assume $d_f < d_w$ and define the local times at $x \in V(G)$ up to the time $n$ as 
     \begin{gather*}
                   L_n (x)  = 
                          \begin{cases}
                                 \frac{1}{m (x)}  \sum_{k=0}^{n-1}  1_{ \{ X_k = x\}  }  &  \text{if $n \ge 1$,}     \\
                                 0                                                                   &  \text{if $n=0$, }                             \end{cases}   
     \end{gather*}
 and the maximum of the local times up to the time $n$ as 
      \begin{gather*}
                   L_n^{\ast} = \sup_{x \in V(G) }    L_n (x) .  
      \end{gather*}

Let $\theta = (d_w - d_f)/2$. 
Recall \eqref{eq:Res} for the definition of the effective resistance, and write 
    $R(x,y) := R_{\text{eff}} ( \{ x \} , \{ y \})$. Also, let $R^{(i)} (x,y) := i^{-(d_w - d_f)} R(x,y)$ for all $x,y \in V(G)$ and $i >0$.   
We first cite a result from \cite{Croydon}.

\begin{lemma} \label{Lem:LocTimeEst20}   
(\cite[Lemma 6.3 (a)]{Croydon}) 
 Suppose Assumption \ref{Ass:rw} (1), (2), (4) and $d_f < d_w$.
 Then there exist constants $c_0, c_1 >0$ such that 
         \begin{gather*}
                   \sup_{i \in [1, \infty)}   \sup_{ \substack{x,y,z \in V(G)  \\ d(x,y) \le  i }  }  
                           P_z \left(  \max_{0 \le k \le  i^{d_w} }   i^{ - (d_w - d_f) } | L_k (x) - L_k (y) |  \ge \lambda \sqrt{R^{(i)} (x,y)}  \right)   
                       \le c_1 \exp (-c_0 \lambda )   
         \end{gather*} 
    for all $\lambda \ge 0$.
In particular, there exist constants $c_1, c_2 >0$ such that  
       \begin{gather}   \label{eq:LTest21}
               \sup_{i \in [1, \infty)}   \sup_{ \substack{x,y,z \in V(G)  \\ d(x,y) \le  i }  }  
                        P_z \left(  \max_{0 \le k \le  i^{d_w} }   | L_k (x) - L_k (y) |  \ge \lambda (id(x,y) )^{\theta} \right)   
                \le c_1 \exp (-c_2 \lambda )   
       \end{gather}
 for all $\lambda \ge 0$.
\end{lemma}

\begin{proof}  
Note that by \cite[Theorem 1.3]{BCK} we have the following relation between the resistance metric and the graph distance
      \begin{gather*} 
                     R(x,y) \asymp  d(x,y)^{d_w - d_f}, \qquad~~\forall x,y\in V(G),  
      \end{gather*} 
 which is a consequence of  Assumption \ref{Ass:rw} (1), (2), (4) and $d_f < d_w$.  
The first statement is the result of \cite[Lemma 6.3 (a)]{Croydon} for the case of $\kappa = T = 1$. 
The second statement can be proved by applying the above relation between the resistance metric and the graph distance.     
\end{proof}

The next theorem is an analogue of \cite[Proposition 4.8]{KKW}. 
Since our proof is different from that of \cite[Proposition 4.8]{KKW} which uses a scaling argument, 
we give the proof below. 
\begin{theorem}[Moduli of continuity of local times]   \label{Thm:ModLocTime}
 Suppose Assumption \ref{Ass:rw} (1), (2), (4) and $d_f < d_w$.
 Then there exist constants $c$, $C>0$ such that 
   \begin{align*}
         &P_o \left(  \max_{ \substack{ x,y \in B_d (o, \kappa u^{1/d_w}) \\  d(x,y)\le L } } \max_{0\le t \le u}  |L_t (x) - L_t (y) |  \ge A \right) 
                               \le c \frac{ (u^{1/d_w} \kappa)^{2d_f}   }{L^{2d_f}}  
                                      \exp  \left( - \frac{CA}{ (\kappa u^{1/d_w} L)^{\theta }   } \right) 
                                         \end{align*}
for all $o \in V(G)$, $u \ge 1$, $\kappa \ge 1$, $0<L \le 2 \kappa u^{1/d_w}$ and $A>0$.
\end{theorem}

\begin{proof}
Let $c_1,c_2$ be as in Lemma \ref{Lem:LocTimeEst20}.  
Let $G^{(i)} $ be a graph with  $V(G^{(i)})  = B_{d} (o,6i)$ and $E(G^{(i)} )  =  \{ (x,y) \in E(G) \mid x,y \in V(G^{(i)}) \}$. 
We denote by $m_i (\cdot ) = m (\cdot \cap V(G^{(i)}) )$ the measure on $G^{(i)}$. 
Then the following holds by the proof of \cite[Theorem 6.1]{Croydon}; 
There exists a positive constant $c_3$ (not depending on $i$) such that 
   \begin{gather}\label{volebi}
                  m_i  ( B_d (x,r) )   \ge c_3 r^{d_f}
   \end{gather}
  for all $i \in [\frac{1}{6}, \infty)$, $x \in V(G^{(i)}) $ and  $r \in [1, 12i]$. 
  (In fact,  \cite[Proof of Theorem 6.1]{Croydon} discusses the case where $i \ge 1$ is an integer,  
       and we can obtain \eqref{volebi} for $i \in [\frac{1}{6},\infty )$ by adjusting the constant.)  
Set $6i= \kappa u^{1/d_w}$. 
By \eqref{volebi}, we have 
         \begin{gather*}
                    \min_{x \in V(G^{(i)})}  m_i (B_{d_i} (x, r) ) = \min_{x \in V(G^{(i)})}  m_i (B_{d} (x, ir) )  \ge c_3 i^{d_f} r^{d_f}, 
         \end{gather*}
where $d_i = \frac{1}{i} d_i$.  
We now apply a discrete version of Garsia's Lemma (see \cite[Proposition 3.1, Remark 3.2]{Croydon})
    to the graph $G^{(i)}$ with distance $d_i = \frac{1}{i} d $, 
    $p(x) = x^{\theta}$,  $\psi (x) = \exp (c_{\ast} |x| ) -1$, 
     and the function $\displaystyle f(x) = \frac{1}{i^{ 2\theta } } L_t (x)$ on $V(G^{(i)})$ for $0 \le t \le u$, 
    where $ c_{\ast} = 12^{-\theta} c_2/2 $. 
For $x,y \in V(G^{ (i)}) = B_d (o,6i)$ with $d(x,y) \le L$ and $t \in [0,u]$, we have 
           \begin{align}
                   \frac{1}{i^{2\theta} } |L_t (x) - L_t (y) |  
                          &\le  \frac{4}{ c_{\ast} }  \int_{0}^{2d_i (x,y)}  4^{\theta} s^{\theta -1} 
                                \log   \left(  \frac{ \Gamma \left( \frac{1}{i^{2\theta } } L_t \right) }{ c_3^2 i^{2d_f}    s^{2d_f}/2^{2d_f}} +1  \right) ds    \notag  \\
                          &\le  \frac{4^{\theta +1}}{c_{\ast}}  \int_{0}^{2L/i}  s^{\theta -1}  
                                \log   \left(  \frac{ \tilde{\Gamma} \left( \frac{1}{i^{2\theta } } L_u \right) }{ c_4 i^{2d_f} s^{2d_f} } +1  \right) ds ,
                                         \label{eq:loctime1}
            \end{align}
 where $c_4 = c_3^2/2^{2d_f}$ and 
          \begin{align*}
                 \Gamma \left( \frac{1}{ i^{2\theta} } L_t \right)  
                          & := \sum_{x,y \in V(G^{(i)})} 
                                 \exp \left( c_{\ast} \frac{ |L_t (x) - L_t (y)| }{ (id(x,y) )^{\theta} } \right) m(x) m(y) ,  \\
                 \tilde{\Gamma}  \left( \frac{1}{ i^{2 \theta} } L_u \right)  
                        & := \sum_{x,y \in V(G^{(i)})} 
                                \exp \left( c_{\ast} \frac{ \sup_{0 \le t \le u} |L_t (x) - L_t (y)|}{ (id(x,y) )^{\theta} } \right) m(x) m(y).  
           \end{align*} 
  Define $\displaystyle v =  \frac{ \tilde{\Gamma} \left( \frac{1}{i^{2\theta} } L_u \right) }{ c_4 i^{2d_f} s^{2d_f} }$. 
  Then by (\ref{eq:loctime1}), we have 
        \begin{align*}
                \frac{1}{i^{2\theta} } |L_t (x) - L_t (y) |    
                 &\le   \frac{4^{\theta +1}}{c_{\ast}} \frac{1}{2d_f} \frac{1}{c_4^{\theta /2d_f}} \frac{1}{i^{\theta}} \tilde{\Gamma} \left( \frac{1}{i^{2\theta} } L_u \right)^{\frac{\theta}{2d_f}} 
                                   \int_{b}^{\infty}   \left(   \frac{1}{v}  \right)^{ \theta/(2d_f) + 1} \log (v+1) dv  \\
                 &=  c_5 \frac{1}{i^{\theta}} \tilde{\Gamma} \left( \frac{1}{i^{2\theta} } L_u \right)^{\frac{\theta}{2d_f}}  \int_{b}^{\infty}   \left(   \frac{1}{v}  \right)^{ \theta/(2d_f) + 1} \log (v+1) dv ,
       \end{align*}
 where $\displaystyle c_5 = \frac{4^{\theta +1}}{c_{\ast}} \frac{1}{ 2d_f} \frac{1}{ c_4^{\theta /2d_f}} $, $\displaystyle  b = \frac{ \tilde{\Gamma} \left( \frac{1}{i^{2 \theta} } L_u \right) }{ c_6 L^{2d_f} }$,
 $\displaystyle c_6 =  2^{2d_f} c_4$. 
By easy calculus we have 
        \begin{align*}
                \int_{b}^{\infty}   \left(   \frac{1}{v}  \right)^{ \theta/(2d_f) + 1} \log (v+1) dv
                         &\le   \frac{ \log (b+1) + 2d_f/\theta }{ \frac{ \theta}{ 2d_f }  \cdot b^{\theta/2d_f } }  .
         \end{align*}
Thus we have 
         \begin{align*}
                 &  \frac{1}{i^{2\theta} } |L_t (x) - L_t (y) |   
                       \le c_7  \left( \frac{L}{i} \right)^{\theta} \left\{ \log (b+1) + \frac{2d_f}{\theta} \right\} .
         \end{align*}
 where $c_7 = c_5 (2d_f / \theta) c_6^{ \frac{\theta}{2d_f} }$, so  
          \begin{align*}
                P_o \left(  \max_{\substack{ x,y \in B_d (o, \kappa u^{1/d_w}) \\  d(x,y)\le L }}  \max_{0\le t \le u}  |L_t (x) - L_t (y) |  \ge A \right)    
                 & \le  P_o \left( \log (b+1)  \ge   \frac{A}{ c_7  (iL)^{\theta} } - \frac{ 2d_f}{\theta }  \right)\\
                 &\le c_8 \left( \frac{E_o \left[ \tilde{\Gamma} \left( \frac{1}{i^{ 2 \theta } } L_u \right) \right] }{c_6L^{2d_f} } + 1 \right)  \exp \left( - c_9 \frac{A}{ (iL)^{\theta } } \right).
       \end{align*}  
By \eqref{eq:LTest21}, noting $ c_{\ast} = 12^{-\theta} c_2/2$, 
       $\kappa \ge 1$ and $6i = \kappa u^{1/d_w}$ (in particular $u \le (6i)^{d_w}$), we have 
            \begin{align*}
                         E_o \left[ \tilde{\Gamma} \left( \frac{1}{i^{ 2 \theta } } L_u \right) \right]  
                          = & \sum_{x,y \in V(G^{(i)} )}   
                                   E_o  \left[  \exp\left( c_{\ast}   \frac{\sup_{0 \le t \le u}  |L_t (x) - L_t (y) | }{ (id(x,y) )^{\theta} } \right)  \right] 
                                   m(x) m(y) \\
                           \le &\sum_{x,y \in V(G^{(i)} )}  \sum_n  \exp ( c_{\ast} (n+1) ) 
                                    P_o \left(  \frac{\sup_{0 \le t \le u}  |L_t (x) - L_t (y) | }{ (id(x,y) )^{\theta} }  \ge n \right) m(x) m(y) \\  
                            \le  & \sum_{x,y \in V(G^{(i)} )}  \sum_n  
                                             \left\{ \sup_{ \substack{x^{\prime}, y^{\prime} \in V(G) \\ d(x^{\prime}, y^{\prime} ) \le 12i }} 
                                                  P_o \left(  \sup_{0 \le t \le (12i)^{d_w}} |L_t (x^{\prime}) - L_t (y^{\prime})| 
                                                    \ge \frac{n}{12^{\theta}} (12id(x^{\prime},y^{\prime}))^{\theta} \right)  \right\}  \\
                                         & \qquad  \times    \exp ( c_{\ast} (n+1) )  m(x)m(y)  \\
                            \le & \sum_{x,y \in V( G^{(i)} )}  \sum_n      \left\{ c_1  \exp (-c_{2} \frac{ n}{12^{\theta}} )  \right\}  \exp ( c_{\ast} (n+1) )  
                                     m(x) m(y) \\  
              \le & c_{10}   i^{2d_f} \le  c_{11}  u^{2d_f/d_w}  \kappa^{2d_f}.
    \end{align*}

Therefore we have 
     \begin{align*}
          P_o \left(  \max_{\substack{ x,y \in B_d (o, \kappa u^{1/d_w}) \\  d(x,y)\le L }}  \max_{0\le t \le u}  |L_t (x) - L_t (y) |  \ge A \right)  
                \le   c_{12} \left(  \frac{  u^{2d_f/d_w}  \kappa^{2d_f} }{  L^{2d_f} }  +1\right)   \exp  \left( - c_{13} \frac{A}{ (\kappa u^{1/d_w} L)^{\theta }   } \right) . 
     \end{align*}   
Since we are assuming $L \le 2 \kappa u^{1/d_w}$, we complete the proof.

\end{proof}

Given Theorem \ref{Thm:ModLocTime}, the following theorem can be proved similarly to
  \cite[Proof of Theorems 4.11 and 4.15]{KKW}. 
(See also \cite[Propositions 4.7 and 4.8]{BassKumagai10}.)  

\begin{theorem}[LILs for the local times]  \label{Thm:LILLoc}
 Suppose Assumption \ref{Ass:rw} (1), (2), (4) and $d_f < d_w$. Then there exist positive constants $c_1 , c_2$ such that the following hold.  
 \begin{align*}
       &\limsup_{n \to \infty}  \frac{L_n^{\ast} }{n^{1-d_s/2}  (\log \log n)^{d_s/2}   }  = c_1,   
                  \qquad  \text{$P_x$-a.s. for $\forall x \in V(G)$, }     \\
       & \liminf_{n \to \infty}  \frac{L_n^{\ast} }{n^{1-d_s/2}  (\log \log n)^{d_s/2 -1 }   }  = c_2,   
                  \qquad  \text{$P_x$-a.s. for $\forall x \in V(G)$.}
 \end{align*}
\end{theorem}

\medskip

Given Theorem \ref{Thm:LILLoc}, the proof of Theorem \ref{Thm:LILdisc} can be done 
similarly as in \cite[Theorem 4.16]{KKW} by using the relation $n = \sum_{x \in R_n} L_n (x) \le R_n L_n^{\ast}$. 
(See also \cite[Propositions 4.9 and 4.10]{BassKumagai10}.) 
\end{appendices}

\vspace{5mm}
\begin{acknowledgements}  
We would like to thank Martin T. Barlow for stimulating discussions when this project was initiated. 
We also thank the anonymous referee for detailed comments and careful corrections.
Especially, we are grateful to the referee to point us out that Theorem \ref{Thm:LILrectra} (II) and 
Lemma \ref{Lem:LIL50} require only 
the on-diagonal heat kernel upper bound \eqref{eq:21dow} rather than the full upper bound \eqref{eq:210}.
The first author was partially supported by JSPS KAKENHI Grant Number 25247007.
The second author was partially supported by JSPS KAKENHI Grant Number 15J02838. 
\end{acknowledgements}

\end{document}